\documentclass{article}
\usepackage[utf8]{inputenc}
\usepackage{authblk}
\usepackage{setspace}
\usepackage[margin=1in]{geometry}
\usepackage{graphicx}
\graphicspath{ {./figures/} }
\usepackage{subcaption}
\usepackage{amsmath}
\usepackage{lineno}
\usepackage{comment}
\usepackage{amsthm}

\newtheorem{thm}{Theorem}[section]

\newtheorem{lem}[thm]{Lemma}

\usepackage{amsfonts} 
\usepackage{mathtools} 

\usepackage{url}
\usepackage{hyperref}
\usepackage{indentfirst}
\usepackage{xcolor}
\usepackage{float}

\hypersetup{
    colorlinks,
    linkcolor={red!50!black},
    citecolor={blue!50!black},
    urlcolor={blue!80!black}
}

\title{Deterministic and Stochastic Studies on Additional Food Provided Prey-Predator Systems with Group Defence among Prey and Mutual Interference among Predators}

\author[1]{D Bhanu Prakash}
\author[2]{D K K Vamsi}

\affil[1, 2]{ \ Department of Mathematics and Computer Science, Sri Sathya Sai Institute of Higher Learning, India.}
\affil[2]{Center for Excellence in Mathematical Biology, Sri Sathya Sai Institute of Higher Learning, India.}
\affil[1]{Corresponding Author. Email: dbhanuprakash@sssihl.edu.in}

\date{}

\begin{document}

\maketitle

\begin{abstract} {{
\noindent The influence of competition and additional food on prey-predator dynamics has attracted considerable interest from mathematical biology researchers in recent times. In this study, we consider an additional food provided prey-predator model exhibiting Holling type-IV functional response among mutually interferring predators. We prove the existence and uniqueness of global positive solutions for the proposed model. We study the existence and stability of equilibrium points and further explored bifurcations with respect to the additional food and competition. We further study the global dynamics of the system and discuss the consequences of providing additional food. Later, we do the time-optimal control studies with respect to the quality and quantity of additional food as control variables by transforming the independent variable in the control system. Making use of the Pontraygin maximum principle, we characterize the optimal quality of additional food and optimal quantity of additional food. We further enhanced the model by incorpoating both continuous and discrete noise. We further characterized and numerically simulated the stochastic optimal controls through Sufficient Stochastic Maximum Principle. We show that the findings of these dynamics and control studies have the potential to be applied to a variety of problems in pest management.
}}
\end{abstract}

{ \bf {keywords:} } Prey-Predator System; Holling type-IV response; Additional Food; Mutual Interference; Bifurcation; Time-optimal control; Brownian Motion; Poisson Process; Pest management;



\section{Introduction} \label{sec:Intro}

\indent In the vast and complex web of life, no species exists in isolation. Every organism is part of a dynamic system, constantly interacting with others and adapting to its environment. Ecology, the study of these interactions, helps us understand how populations—groups of the same species living in a shared habitat—coexist, compete, and evolve. Species interactions can be categorized based on how each participant is affected. Competition arises when species vie for limited resources, often leading to reduced success for both. In contrast, mutualism benefits both parties, such as bees pollinating flowers. Predation benefits one species at the expense of another, while commensalism benefits one without affecting the other. These interactions shape the structure and function of ecosystems. 


This study focuses on a one prey–one predator system, with a particular emphasis on Mutual Interference—a form of competition where predators hinder each other while hunting the same prey. Understanding this interaction has valuable implications for biological conservation and pest management, helping refine strategies to protect ecosystems and control invasive species effectively.

Alfred J. Lotka \cite{lotka1925elements} and Vito Volterra \cite{volterra1927variazioni} proposed the following first mathematical model for interacting species in $1925$. 
\begin{eqnarray*}
	\frac{\mathrm{d} N}{\mathrm{d} T} &=& r N - a N P , \\
	\frac{\mathrm{d} P}{\mathrm{d} T} &=& - m P + b N P.
\end{eqnarray*}

Here $r$ represents the prey growth rate (Malthus law); $m$ represents the predator mortality rate (Mass action law); $a, \ b$ stands for the predation coefficients with $b = e a$, where $e$ is the conversion coefficient. This model describes the traditional concept of prey-predation cycle: The predator populations crash shortly after their prey populations experience a crash; This cyclic behavior is observed in many species like Canadian Lynx vs Snowshoe hare. Ever since Lotka and Volterra proposed the first mathematical model, among interacting species, the field of mathematical ecology has flourished with diverse models offering valuable insights. 

The generic model of two interacting populations, specifically, populations in competition for a limiting resource, is given by
\begin{eqnarray}
	\frac{\mathrm{d} N}{\mathrm{d} T} &=& f(N) - h(N,P),  \nonumber \\
	\frac{\mathrm{d} P}{\mathrm{d} T} &=& - g(P) + e h(N,P). \label{genmodel}
\end{eqnarray}

Here, $N$ and $P$ stands for prey and predator populations respectively. The functions $f(N), \ g(P),\ h(N,P)$ stands for the prey growth term, predator death term and the predation behavior respectively. The parameter $e$ is the conversion factor of prey to predator biomass.

In this study, we incorporate the group defence among prey using the Holling type-IV functional response or the Monod-Haldane functional response \cite{metz2014dynamics,kot2001elements}. This response is displayed by many organisms in nature \cite{yano2012cooperative, mcclure2011defensive}. In addition to this, We also consider the provision of additional food to predators, which is considered as the product of it's quality and quantity. Numerous real-world examples highlight the importance of examining the quality of additional food \cite{magro2004comparison, vandekerkhove2010pollen, winkler2005plutella} as well as its quantity \cite{vandekerkhove2010pollen, put2012type, urbaneja2013sugar}. Some of these features like additional food, group defence and mutual intereferene are studied in isolation. However, to the best of our knowledge, this work is an initial attempt to capture all these three characteristics in the same model.

In this study, we perform the deterministic and stochastic dynamics and control studies on the proposed model. The global dynamics and bifurcations of Holling type-II functional responses have been extensively studied in the literature \cite{Islam2024, Type4-7, Fractional3, Type4-9, AAAA1}. Several studies have also explored models that exhibit functional responses other than Holling’s \cite{AAAA3, AAAA9, AAAA7, AAAA5}. The deterministic and stochastic dynamics of a prey-predator model with a Beddington-De Angelis functional response is studied in \cite{AAAA1}. The effect of additional food on natural enemies exhibiting Holling type-II functional responses is well-documented in \cite{VP1,VP3,VP2, Application1, AAAA2}. The authors in \cite{V3EarthSystems} and \cite{V3JTB} have studied additional food provided prey-predator systems involving type-III and type-IV functional responses, respectively. 


A large number of researchers introduced a stochastic environmental variation using the Brownian motion into parameters in the deterministic model to construct a stochastic model \cite{oksendal2013}. The authors in \cite{Stochastic6} proved that a stochastic prey - predator model with a protection zone has a unique stationary distribution which is ergodic. In  \cite{Stochastic8} the authors obtained stochastic permanence for a stochastic prey - predator system with Holling-type III functional response. The work \cite{Stochastic7} deals with a Holling-type II stochastic prey - predator model with additional food that has an ergodic stationary distribution. The authors in \cite{Stochastic5} studied the deterministic and stochastic dynamics of a modified Leslie-Gower prey - predator system with a simplified Holling-type IV scheme. The work \cite{Type3-3} deals with the  survival and ergodicity of a stochastic Holling-type III prey - predator model with markovian switching in an impulsive polluted environment. Authors in \cite{Stochastic5, Stochastic4} studied the deterministic and stochastic dynamics of a modified Leslie-Gower prey-predator system with simplified Holling type-IV functional response. \cite{Levyjumps2} uses the stochastic averaging method to analyze the modified stochastic Lotka-Volterra models under combined Gaussian and Poisson noise. \cite{Levyjumps} studies the dynamics and dynamics of a Stochastic One-Predator-Two-Prey time delay system with jumps.

In order to avoid extinction or drastic changes in the population of various species, control measures are introduced to this system. Among the various control interventions, the control intervention that is focused in this work is the additional food for predator. Mathematically, control problems are classified into various problems like Bolza problem, Lagrange problem, Mayer problem, Time-Optimal Problem \cite{liberzon2011calculus, evans1983introduction, nisio2015stochastic, fleming2012deterministic}.  In addition to this, optimal control studies on various prey-predator models with Holling type-III and type-IV responses are thoroughly investigated in \cite{ANT3Quality,ANT3Quantity,ANT4Quality,ANT4Quantity}. The optimal control studies on various prey-predator models with Holling type II response are thoroughly investigated in \cite{ControlDelay, TimeOptimalImpulse1, OptimalControlANN, OptimalControlStochastic,VP3, VP2}. In the stochastic setting, optimal control problems like Lagrange problem are studied in various works \cite{sss1, sss2, Harvesting1}. Very little work is available on the stochastic time-optimal control problems \cite{Theory8, Theory10, TimeOptimalStochastic, Theory11}. 

In this study, we extend the generalized prey-predator model (\ref{genmodel}) by incorporating three key ecological features. First, we model group defense among prey using a Holling type-IV functional response or the Monod-Haldane functional response \cite{metz2014dynamics,kot2001elements}, which captures the reduction in predation efficiency at higher prey densities. Second, we account for mutual interference among predators by modifying the functional response to include the time predators spend on both searching for and handling prey, as well as on interactions with other predators. Third, we introduce additional food provision for predators, modeled as the product of its quality and quantity. Such food supplementation has been recognized as a valuable tool in ecological conservation and pest management. By integrating these components, we develop a comprehensive prey-predator system that reflects group defense in prey, mutual interference and provision of additional food among predators.

The article is structured as follows: Section \ref{sec:4midmodel} derives the functional response with additional food, prey defence and mutual interference and formulates the prey-predator model with these functional responses. Section \ref{sec:4midposi} proves the positivity and boundedness of the solutions of the proposed system. Section \ref{sec:4midequil} investigates the conditions for the existence of various equilibria. The local stability of these equilibria is presented in section \ref{sec:4midstab}. In section \ref{sec:4midbifur}, we present the various possible local bifurcations exhibited by the proposed model both analytically and numerically. Section \ref{sec:4midglobaldynamics} studies the global dynamics of the proposed system in the parameter space of quality and quantity of additional food. Section \ref{sec:4midconseq} provides a detailed analysis of the consequences of providing additional food. Section \ref{sec:4midtimecontrol} presents the theoretical and numerical studies on time-optimal control problems with quality or quantity of additional food as control parameters. Section \ref{sec:4mismodel} extends the deterministic model with continuous and discrete noise. Section \ref{sec:4miscontrol} defines the time-optimal control problem and provides analytical and numerical solutions for the control problem with two controls and two noises. Finally, we present the discussions and conclusions in section \ref{sec:disc}.

\section{Model Formulation} \label{sec:4midmodel}

\indent In this section, we formulate the mathematical model to study the consequences of providing additional food to the prey-predator interactions exhibitng Holling type-IV functional response among mutual interfering predators. 

In the present study, we model the prey growth in the absence of predators using logistic equation. In the presence of predators, we assume that the predation behavior follows Holling type-IV functional response. In addition, we incorporate the mutual interference behavior of the predators in this model. We also assume that the additional food supplemented to predators is distributed uniformly in the habitat. The functional response of predator towards prey and additional food is given in section \ref{4midsecfr}.  

\subsection{Derivation of Holling Type-IV Functional Response} \label{4midsecfr}

Let $N,\ P$ denote the densities of prey and predator respectively. 

Let $\triangle T$, a small period of time (small in the sense that the predator and prey densities remain roughly constant over $\triangle$T) that the predator spends for searching prey and/or additional food, consuming the captured prey and/or additional food, and interacting with other predators.

Let $\triangle T_S$ denote the part of $\triangle T$ that the predator spends for searching prey and/or additional food. 

Let $\triangle T_N$ denote the part of $\triangle T$ that the predator spends for handling the prey.

Let $\triangle T_A$ denote the part of $\triangle T$ that the predator spends for handling the additional food. 

Let $\triangle T_P$ denote the part of $\triangle T$ that the predator spends for interacting with other predators.

So, we have
\begin{equation} \label{4middelt}
	\triangle T = \triangle T_S + \triangle T_N + \triangle T_A + \triangle T_P.
\end{equation}

Let $h_N, h_A$ and $h_P$ denote the length of time required for each interaction between a predator and prey item, additional food, other predators respectively.

Let $e_N$ represents the search rate of the predator per unit prey availability and $e_A$ represents the search rate of the predator per unit quantity of additional food. Let $e_P$ represents the rate constant at which a predator encounters other predators.

Let $b$ represents the inhibitory effect or the group defense of the prey. Then, from \cite{V4DEDS}, the total handling time for the prey ($\triangle T_N$) equals $h_N$ times number of prey caught, which is given by, $h_N \frac{e_N}{bN^2+1} N \triangle T_S$. 

Now, the handling time for the additional food equals the handling time for one additional food item times the total density of additional food encountered. Also,the additional food encountered is proportional to the search time and the additional food density. Hence, the quantity of additional food encountered during this time is proportional to $\triangle T_S  A = e_A T_S A$, where $e_A$ is the proportionality constant, which denotes the catchability of the additional food. Now, the total handling time for the additional food($\triangle T_A$)  equals $h_A$ times additional food encountered, which is given by, $h_A e_A A \triangle T_S$. 

Similarly, $\triangle T_P$ is given by $h_P e_P P  \triangle T_S$.

Hence, 

\begin{eqnarray}
	\triangle T_N &=& h_N \frac{e_N}{bN^2+1}  N \triangle T_S. \label{4midtn} \\
	\triangle T_A &=& h_A e_A A  \triangle T_S. \label{4midta} \\
	\triangle T_P &=& h_P e_P P  \triangle T_S. \label{4midtp} 
\end{eqnarray}

From (\ref{4middelt}), (\ref{4midtn}), (\ref{4midta}) and (\ref{4midtp}), we have

\begin{equation} \label{4middeltt}
	\triangle T = \left(1 + h_N \frac{e_N}{bN^2+1} N + h_A e_A A  + h_P e_P P \right) \triangle T_S.
\end{equation}

Now, as a predator encounters $N \frac{e_N}{bN^2+1} \triangle T_S$ prey items during the period $\triangle T$, the overall rate of encounters with prey over the time interval $\triangle T$ is given by

\begin{equation*}
	\begin{split}
		g(N,P,A) & = \frac{\text{Total number of Prey Caught}}{\text{Total Time}}\\
		& = \frac{\frac{e_N}{bN^2+1} N \triangle T_S}{\triangle T} = \frac{\frac{e_N}{bN^2+1} N \triangle T_S}{(1 + h_N \frac{e_N}{bN^2+1} N + h_A e_A A  + h_P e_P P ) \triangle T_S} \\
		& = \frac{\frac{e_N}{bN^2+1} N}{1 + h_N \frac{e_N}{bN^2+1}  N + h_A e_A A  + h_P e_P P } \\
		& = \frac{e_N N}{(bN^2+1)(1 +  + h_A e_A A  + h_P e_P P) + h_N e_N N}.
	\end{split}
\end{equation*}
Dividing numerator and denominator by $e_N h_N$

\begin{equation*}
	\begin{split}
		g(N,P,A) & = \frac{\frac{1}{h_N} N}{(bN^2+1)\left(\frac{1}{e_N h_N}+\frac{h_A e_A}{h_N e_N}A+\frac{h_P e_P}{h_N e_N} P\right)+N}.
	\end{split}
\end{equation*}

Similarly, as a predator encounters $e_A A \triangle T_S$ quantity of additional food during the period $\triangle T$, the overall rate of encounters with additional food over the time interval $\triangle T$ is given by

\begin{equation*}
	\begin{split}
		h(N,P,A) & = \frac{\text{Total Additional Food Encountered}}{\text{Total Time}}\\
		& = \frac{e_A A \triangle T_S}{\triangle T} = \frac{e_A A \triangle T_S}{(1 + h_N  \frac{e_N}{bN^2+1} N  + h_A e_A A  + h_P e_P P ) \triangle T_S} \\
		& = \frac{ e_A A}{1 + h_N \frac{e_N}{bN^2+1} N  + h_A e_A A  + h_P e_P P } \\
		& = \frac{e_A A (b N^2+1)}{(bN^2+1)(1 +  + h_A e_A A  + h_P e_P P) + h_N e_N N}.
	\end{split}
\end{equation*}
Dividing numerator and denominator by $e_N h_N$

\begin{equation*}
	\begin{split}
		h(N,P,A) & = \frac{\frac{e_A}{e_N h_N} A (b N^2 + 1)}{(bN^2+1)\left(\frac{1}{e_N h_N}+\frac{h_A e_A}{h_N e_N}A+\frac{h_P e_P}{h_N e_N} P\right)+N}.
	\end{split}
\end{equation*}

Hence, the Holling type-IV functional response of the mutually interfereing predator towards the prey ($g(N,P,A)$) and the additional food ($h(N,P,A)$) respectively are given by
\begin{eqnarray}
	g(N,P,A) = \frac{\frac{1}{h_N} N}{(bN^2+1)(\frac{1}{e_N h_N}+\frac{h_A e_A}{h_N e_N}A+\frac{h_P e_P}{h_N e_N} P)+N}. \label{4midg} \\
	h(N,P,A) = \frac{\frac{e_A}{e_N}\frac{1}{h_N} A (b N^2 + 1)}{(bN^2+1)(\frac{1}{e_N h_N}+\frac{h_A e_A}{h_N e_N}A+\frac{h_P e_P}{h_N e_N} P)+N}. \label{4midh}
\end{eqnarray}

\subsection{Derivation of Model} \label{4midsecderm}
Now, we derive the additional food provided mutually interfering prey-predator model representing the prey-predator dynamics when the predator is provided with additional food (which is of non-reproducing item) based on the functional response derived in the section \ref{4midsecfr}. 

The deterministic prey-predator model exhibiting Holling type-IV functional response among mutually interfering predators and where prey grow logistically is given by
\begin{equation*}
	\begin{split}
		\frac{\mathrm{d}N}{\mathrm{d}T} & = r N \left(1-\frac{N}{K}\right) - g(N,P,A) P, \\
		\frac{\mathrm{d}P}{\mathrm{d}T} & = \left(\epsilon_N g(N,P,A)+\epsilon_A h(N,P,A) \right)P - m_1 P. \\
	\end{split}
\end{equation*}

Here the parameters $r$ and $K$ represent the intrinsic growth rate and carrying capacity of the prey respectively. $m_1$ is the death rate of predator in the absence of prey which is also termed as starvation rate. $\epsilon_N,\ \epsilon_A$ are the conversion factors that represents the rate at which the prey biomass, additional food biomass gets converted into predator biomass respectively.

Substituting (\ref{4midg}) and (\ref{4midh}) in the above equations, we get 
\begin{equation*}
	\begin{split}
		\frac{\mathrm{d}N}{\mathrm{d}T} & = r N \left(1-\frac{N}{K}\right) -  \frac{\frac{1}{h_N} N}{(bN^2+1)(\frac{1}{e_N h_N}+\frac{h_A e_A}{h_N e_N}A+\frac{h_P e_P}{h_N e_N} P)+N} P, \\
		\frac{\mathrm{d}P}{\mathrm{d}T} & = \Bigg( \epsilon_N  \frac{\frac{1}{h_N} N}{(bN^2+1)(\frac{1}{e_N h_N}+\frac{h_A e_A}{h_N e_N}A+\frac{h_P e_P}{h_N e_N} P)+N} \\
		& \ \  + \epsilon_A \frac{\frac{e_A}{e_N}\frac{1}{h_N} A (b N^2 + 1)}{(bN^2+1)(\frac{1}{e_N h_N}+\frac{h_A e_A}{h_N e_N}A+\frac{h_P e_P}{h_N e_N} P)+N} \Bigg) P - m_1 P. \\
	\end{split}
\end{equation*}

Here, let $c = \frac{1}{h_N}$ and $a = \frac{1}{h_N e_N}$ stands for the maximum rate of predation and half-saturation value of the predator respectively. 

The term $\alpha = \frac{\epsilon_N/h_N}{\epsilon_A/h_A}$, which is the ratio between the maximum growth rates of the predator when it consumes the prey and additional food respectively, indicates the relative efficiency of the predator to convert either of the available food into predator biomass. The value $\alpha$ can be seen to be an equivalent of \textbf{quality} of additional food. Let $\eta = \frac{e_A \epsilon_A}{e_N \epsilon_N}$ and the term $\eta A$ effectual food level. Here the term $\epsilon_1 = e_P h_P$ represents the strength of mutual interference between predators. So, by substituting $\epsilon_A \frac{e_A}{e_N} = \epsilon_N \eta, \ \frac{h_A e_A}{h_N e_N} = \alpha \eta, \ \delta_1 = \epsilon_N c$ and $\epsilon_1 = h_P e_P$, the model gets transformed to 
\begin{equation}\label{4midtemp}
	\begin{split}
		\frac{\mathrm{d}N}{\mathrm{d}T} & = r N \left(1-\frac{N}{K}\right) -  \frac{cNP}{(bN^2+1)(a+\alpha \eta A+ \epsilon_1 a P)+N}, \\
		\frac{\mathrm{d}P}{\mathrm{d}T} & = \frac{\delta_1 \left(N+\eta A(bN^2+1)\right) P}{(bN^2+1)(a+\alpha \eta A+ \epsilon_1 a P)+N} - m_1 P. \\
	\end{split}
\end{equation}

The biological descriptions of the various parameters involved in the system (\ref{4midtemp}) is described in \autoref{4midparam_tab}.

\begin{table}[bht!]
	\centering
	\begin{tabular}{ccc}
		\hline
		Parameter & Definition & Dimension \\  
		\hline
		T & Time & time\\ 
		N & Prey density & biomass \\
		P & Predator density & biomass \\
		A & Additional food & biomass \\
		r & Prey intrinsic growth rate & time$^{-1}$ \\
		K & Prey carrying capacity & biomass \\
		c & Maximum rate of predation & time$^{-1}$ \\
		$\delta_1$ & Maximum growth rate of predator & time$^{-1}$ \\
		$m_1$ & Predator mortality rate & time$^{-1}$ \\
		b & Group defence in prey & biomass$^{-2}$ \\
		\hline
	\end{tabular}
	\caption{Description of variables and parameters present in the system (\ref{4midtemp})}
	\label{4midparam_tab}
\end{table}

In order to reduce the complexity of the model, we non-dimensionalize the system (\ref{4midtemp}) using the following transformations
$$t=rT,\ N=ax, \  P=\frac{ary}{c}. $$

Accordingly, system (\ref{4midtemp}) gets reduced to the following non-dimensionalised systems respectively :

\begin{equation} \label{4mid}
	\begin{split}
		\frac{\mathrm{d} x}{\mathrm{d} t} & = x \left(1-\frac{x}{\gamma} \right)- \frac{xy}{(\omega x^2 + 1) (1+\alpha \xi+\epsilon y)+x}, \\
		\frac{\mathrm{d} y}{\mathrm{d} t} & = \frac{\delta \left(x + \xi (\omega x^2 + 1) \right) y}{(\omega x^2 + 1) (1+\alpha \xi+\epsilon y)+x} - m y. \\
	\end{split}
\end{equation}
where $$\gamma = \frac{K}{a}, \ \xi = \frac{\eta A}{a}, \  \epsilon = \frac{\epsilon_1 a r}{c}, \ \delta = \frac{\delta_1}{m},\ m_1 = r m, \ \omega = b a^2.$$

Here the term $\frac{\eta A}{N}$ denotes the quantity of additional food perceptible to the predator with respect to the prey relative to the nutritional value of prey to the additional food. Hence the term $\xi = \frac{\eta A}{a^2}$ can be seen to be an equivalent of {\textit {\bf{quantity}}} of additional food.

\section{Positivity and boundedness of the solution} \label{sec:4midposi}

\subsection{Positivity of the solution}

In this section, we demonstrate that the positive $xy$-quadrant is an invariant region for the system (\ref{4mid}). Specifically, this means that if the initial populations of both prey and predator start in the positive $xy$-quadrant (i.e., $x(0) > 0$ and $y(0) > 0$), they will remain within this quadrant for all future times.

If prey population goes to zero (i.e., $x(t)=0$), then it is observed from the model equations (\ref{4mid}) that $\frac{\mathrm{d} x}{\mathrm{d} t} = 0.$ This means that the prey population is constant (remains at zero) and cannot be negative. This holds even for the case when predator population goes to zero (i.e., $y=0$). Notably, $x=0$ and $y=0$ serve as invariant manifolds, with $\frac{\mathrm{d} x}{\mathrm{d} t} \Big|_{x=0} = 0$ and $\frac{\mathrm{d} y}{\mathrm{d} t} \Big|_{y=0} = 0$. Therefore, if a solution initiates within the confines of the positive $xy$-quadrant, it will either remains positive or stays at zero eternally (i.e., $x(t) \geq 0$ and $y(t) \geq 0 \ \forall t>0$ if $x(0)>0$ and $y(0)>0$). 

\subsection{Boundedness of the solution}

\begin{thm}
	Every solution of the system (\ref{4mid}) that starts within the positive quadrant of the state space remains bounded. \label{4midbound}
\end{thm} 

\begin{proof}
	We define $W = x + \frac{1}{\delta}y$. Now, for any $K > 0$, we consider,
	\begin{equation*}
		\begin{split}
			\frac{\mathrm{d} W}{\mathrm{d} t} + K W = & \frac{\mathrm{d} x}{\mathrm{d} t} + \frac{1}{\delta} \frac{\mathrm{d} y}{\mathrm{d} t} + K x + \frac{K}{\delta} y \\
			=  & x \left(1-\frac{x}{\gamma} \right)- \frac{xy}{x + (\omega x^2 + 1)(1+\alpha \xi+\epsilon y)} \\ 
			& + \frac{1}{\delta} \left( \delta \left( \frac{x + \xi (\omega x^2 + 1)}{x + (\omega x^2 + 1)(1 + \alpha \xi + \epsilon y)} \right) y - m y \right) + K x + \frac{K}{\delta} y \\
			=  & x -\frac{x^2}{\gamma} - \frac{xy}{x + (\omega x^2 + 1)(1+\alpha \xi+\epsilon y)} \\ 
			& +  \left( \frac{x + \xi (\omega x^2 + 1)}{x + (\omega x^2 + 1)(1 + \alpha \xi + \epsilon y)} \right) y - \frac{m}{\delta} y + K x + \frac{K}{\delta} y \\
			=  & (1+K) x -\frac{x^2}{\gamma} + \frac{\xi (\omega x^2 + 1) y}{x + (\omega x^2 + 1)(1 + \alpha \xi + \epsilon y)} + \frac{K-m}{\delta} y. \\
			\text{Since $x \geq 0$,} & \\
			\leq  & (1+K) x - \frac{x^2}{\gamma} + \frac{\xi (\omega x^2 + 1) y}{(\omega x^2 + 1)(1 + \alpha \xi + \epsilon y)} + \frac{K-m}{\delta} y \\
			=  & (1+K) x - \frac{x^2}{\gamma} +  \frac{\xi y}{1 + \alpha \xi + \epsilon y} + \frac{K-m}{\delta} y \\
			\leq  & \frac{\gamma (1+K)^2}{4} +  \frac{\xi}{\epsilon } + \frac{K-m}{\delta} y. \\
		\end{split}
	\end{equation*}
	
	For sufficiently small $K(<m)$, we get
	$$\frac{\mathrm{d} W}{\mathrm{d} t} + K W \leq M \left(= \frac{\gamma (1+K)^2}{4} +  \frac{\xi}{\epsilon} \right).$$
	
	Using Gronwall's inequality \cite{howard1998gronwall}, we now find an upper bound on $W(t)$. 
	
	This inequality is in the standard linear first-order form, and we solve it by multiplying both sides by an integrating factor $e^{Kt}$. This simplify the above inequality:
	$$\frac{\mathrm{d}}{\mathrm{d} t} (W(t) e^{Kt}) \leq M e^{Kt}.$$
	
	Now, integrating both sides from $0$ to $t$, we get
	
	$$0 \leq W(t) \leq \frac{M}{K} \left(1 - e^{-Kt}\right) + W(0)\  e^{-Kt}.$$
	
	Therefore, $0 < W(t) \leq \frac{M}{K}$ as $t \rightarrow \infty$. This demonstrates that the solutions of system (\ref{4mid}) are ultimately bounded, thereby proving Theorem \autoref{4midbound}.
\end{proof}

The \text{Picard-Lindel\"of theorem} guarentees the existence of a unique solution that exists locally in time for the system (\ref{4mid}), given any initial conditions $x(0) = x_0 > 0$ and $y(0) = y_0 > 0$. This happens because the RHS terms in (\ref{4mid}) are continuous and locally Lipschitz. Since the solution does not blow up in finite time (i.e., that the solution exists for all $t \geq 0$), global existence is also guaranteed.

\section{Existence of Equilibria} \label{sec:4midequil}

In this section, we investigate the existence of various equilibria that system (\ref{4mid}) admits and study their stability nature. We first discuss the nature of nullclines of the considered system and the asymptotic behavior of its trajectories. We consider the biologically feasible parametric constraint $\delta > m$. 

The prey nullclines of the system (\ref{4mid}) are given by 
$$ x = 0 , \  \ 1- \frac{x}{\gamma}  - \frac{y}{(\omega x^2 + 1)(1+\alpha \xi+\epsilon y) + x} = 0. $$

The predator nullclines of the system (\ref{4mid}) are given by
$$ y = 0, \ \ \frac{\delta \left(x + \xi (\omega x^2 + 1)\right)}{(\omega x^2 + 1)(1 + \alpha \xi + \epsilon y) + x} - m = 0. $$

Upon simplification, the non trivial prey nullcline is given as
\begin{equation} \label{4midprey}
	y = \frac{\left(1-\frac{x}{\gamma}\right)\left[(\omega x^2 + 1)(1+\alpha \xi) + x\right]}{1 - \epsilon \left(1-\frac{x}{\gamma}\right) (\omega x^2 + 1) }.
\end{equation}

In the absence of mutual interference (i.e., $\epsilon = 0$), it is a cubic equation that is positive when $x \in [0,\gamma)$ and is negative when $x \in (\gamma,\infty)$. When mutual interference is incorporated, this nullcline is no more a smooth curve. It passes through the point $(\gamma, 0)$. It also passes through $(0,\frac{1+\alpha \xi}{1-\epsilon})$ and touches the positive $y$-plane if $0 \leq \epsilon < 1$. 

This prey nullcline goes to infinity when the denominator goes to zero. Let 
\begin{equation} \label{eps1eqn}
	f(x) = \frac{\epsilon \omega}{\gamma} x^3 - \epsilon \omega x^2 + \frac{\epsilon}{\gamma} x + 1 - \epsilon.
\end{equation}

The first and second derivatives are given by:
\begin{eqnarray*}
	f'(x) &=& \frac{3\epsilon \omega}{\gamma} x^2 - 2 \epsilon \omega x + \frac{\epsilon}{\gamma}. \\
	f''(x) &=& \frac{6\epsilon \omega}{\gamma} x - 2 \epsilon \omega.
\end{eqnarray*}

At $x=0$, we have $f(0) = 1 - \epsilon,\ f'(0) = \frac{\epsilon}{\gamma} > 0,\ f''(0) = - 2 \epsilon \omega < 0$. 

At $x=\gamma$, we have $f(\gamma) = 1 >0,\ f'(\gamma) = \epsilon \omega \gamma + \frac{\epsilon}{\gamma} > 0,\ f''(\gamma) = 4 \epsilon\omega>0$.

Also at $x = \frac{\gamma}{3}$, we have $f''(x) = 0$. Therefore, $x = \frac{\gamma}{3}$ is the inflection point of $f(x)$. 

We now have the following cases describing the existence of asymptotes for prey nullcline in the interval $(0,\gamma)$.

\begin{itemize}
	\item \textit{Case I - $0 <\epsilon<1$}: In this case, $0 < f(0) = 1 - \epsilon < 1 = f(\gamma)$. Therefore, prey nullcline has $0$ or $2$ roots in the interval $[0,\gamma]$. Hence when $0 \leq \epsilon < 1$, the prey nullcline goes to $\infty$ either twice (say, case A) or none. When prey nullcline never goes to $\infty$, it can monotonically decrease from $(0,\frac{1+\alpha \xi}{1-\epsilon})$ to $(\gamma, 0)$ (say, case B) or exhibit crest and trough between these two points (say, case C). 
	\item \textit{Case 2 - $\epsilon = 1$}: When $\epsilon = 1$, the roots of $f(x)$ are $0, \frac{\gamma \pm \sqrt{\gamma^2 - \frac{4}{\omega}}}{2}$. When $\gamma^2 \omega < 4$, the nullcline goes to $\infty$ only near $0$ (say, case D). Else, $f(x)$ has three roots and the prey nullcline goes to $\infty$ at three values in the range $[0,\gamma]$ (say, case E).
	\item \textit{Case 3 - $\epsilon > 1$}: When $\epsilon > 1$, we have $f(0) =  \frac{\gamma (1 - \epsilon)}{\epsilon \omega} < 0$ and $f(\gamma) = \frac{\gamma}{\epsilon \omega} > 0$. This implies that it has 1 or 3 roots in the interval $[0,\gamma]$. We also have $f'(0) > 0$, $f'(\gamma) > 0$ and $f'(x) = 0$ when $x = \frac{\gamma \pm \sqrt{\gamma^2 - \frac{3}{\omega}}}{3}$. Note that both the extremes are in the range $[0,\gamma]$.  
	
	\begin{figure}[ht]
		\centering
		\includegraphics[width=\textwidth]{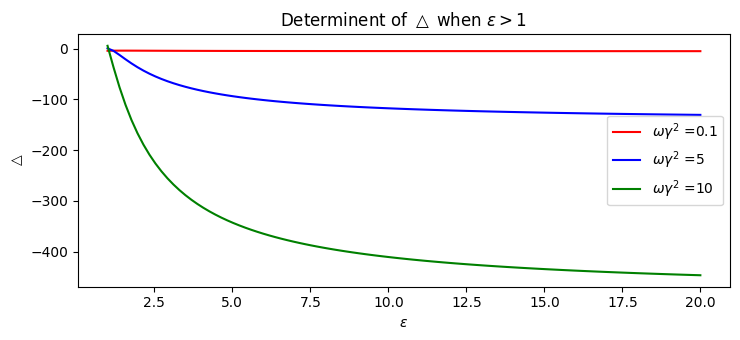}
		\caption{The determinent of the cubic equation (\ref{eps1eqn}) when $\epsilon > 1$. }
		\label{eps1fig4mid}
	\end{figure}
	
	Also, the discriminant of $f(x)$ is given as $\triangle = \frac{1}{\epsilon^2 \omega^3} [(-4 - 8 a - 4 a^2) \epsilon^2 + (36 a + 4 a^2) \epsilon - 27 a]$, where $a = \omega \gamma^2$. Upon numerically simulating the value of $\triangle$ for various values of $a=\omega \gamma^2$, it is observed from figure \ref{eps1fig4mid} that the discriminant is always negative. This tells us that there is only one positive root in the interval $[0,\gamma]$. Hence when $\epsilon > 1$, the prey nullcline goes to $\infty$ only at one point in $[0,\gamma]$ (say, case F).
\end{itemize}

The non trivial predator nullcline is given as
$$\frac{\delta \left( x + \xi (\omega x^2 + 1) \right)}{x + (\omega x^2 + 1)(1 + \alpha \xi + \epsilon y)} - m  = 0.$$ 

In the absence of mutual interference (i.e., $\epsilon = 0$), these nullclines are straight lines parallel to $y$-axis. Upon simplification, the non trivial predator nullcline is
\begin{equation} \label{4midpredator}
	y = \frac{(\delta - m) x + \left(\delta \xi - m (1+\alpha \xi) \right) \left( \omega x^2 + 1 \right)}{m \epsilon (\omega x^2 + 1)}.
\end{equation}

This predator nullcline passes through $(0,\frac{\delta \xi - m (1+\alpha \xi)}{m \epsilon})$. This point will be on the positive-$y$ axis if $\delta \xi - m (1+\alpha \xi) > 0$. 

Also this nullcline touches the $x$-axis at $(\frac{-(\delta - m) \pm \sqrt{\triangle}}{2 \omega (\delta \xi - m(1+\alpha \xi))},0)$ where $\triangle  = (\delta - m)^2 - 4 \omega (\delta \xi - m(1+\alpha \xi))^2$. This point will be on the positive-$x$ axis if $\triangle > 0$ and $\delta \xi - m (1+\alpha \xi) < 0$. These two conditions can be summarised as $\frac{-(\delta - m)}{2 \sqrt{\omega}} < \delta \xi - m (1+\alpha \xi) < 0$. It is also observed that the predator nullcline never reaches $\infty$ in finite time.

The slope of the predator nullcline is given by
$$\frac{\mathrm{d} y}{\mathrm{d} x} =   \left( \frac{\delta - m}{m \epsilon} \right)  \frac{1-\omega x^2}{(1+\omega x^2)^2}.$$

The slope is positive when $x=0$ and  the slope is $0$ when $x = \pm \frac{1}{\sqrt{\omega}}$. 

The second derivative is given by 
$$\frac{\mathrm{d}^2 y}{\mathrm{d} x^2} =   \left( \frac{\delta - m}{m \epsilon} \right) \frac{2 \omega x ( \omega x^2 - 3)}{(1 + \omega x^2)^3}.$$

From the second derivative test, predator nullcline attains maximum at $x=\frac{1}{\sqrt{\omega}}$ and the maximum value is given as $\left( \frac{\delta - m}{m \epsilon} \right) \frac{1}{2\sqrt{\omega}} + \frac{\delta \xi - m (1+\alpha \xi)}{m \epsilon}  > 0$ i.e., $\delta \xi - m (1+\alpha \xi) > \frac{-(\delta - m)}{2 \sqrt{\omega}}$.

Hence we have the sufficient condition for the existence of predator nullcline in the first quadrant is 
\begin{equation} \label{condition14mid}
	\delta \xi - m (1+\alpha \xi) > \frac{-(\delta - m)}{2 \sqrt{\omega}}.
\end{equation}

The qualitative behavior of the predator nullcline can be understood in two scenarios:
\begin{itemize}
	\item \textbf{Case 1 $\left( \delta \xi - m (1+\alpha \xi) > 0 \right)$:} Predator nullcline touches only positive $y$-axis and not pass through positive $x$-axis. 
	\item \textbf{Case 2 $\left( 0 > \delta \xi - m (1+\alpha \xi) > \frac{-(\delta - m)}{2 \sqrt{\omega}} \right)$: } Predator nullcline touches only positive $x$-axis and not pass through positive $y$-axis. 
\end{itemize}

The possible configurations for the A-F cases of prey nullcline with case 1 and case 2 of predator nullclines are presented in \autoref{intereqcase14mid} and \autoref{intereqcase24mid} respectively. In these figures, the solid blue and green lines are prey nullcline and predator nullclines respectively. From these two figures, it is observed that the interior equillibrium point exists for the system (\ref{4mid}).

\begin{figure}[htbp]
	\centering
	\begin{minipage}{\textwidth}
		\centering
		\includegraphics[width=0.8\textwidth]{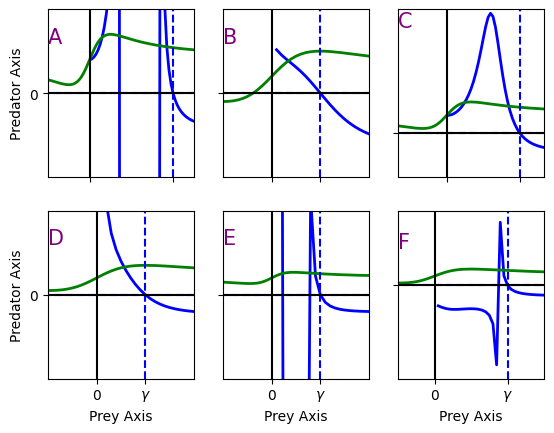}
		\caption{The possible configurations for the prey and predator nullclines of the system (\ref{4mid}) when $\delta \xi - m (1+\alpha \xi) > 0$.}
		\label{intereqcase14mid}
	\end{minipage}
	
	\vspace{1em}
	
	\begin{minipage}{\textwidth}
		\centering
		\includegraphics[width=0.8\textwidth]{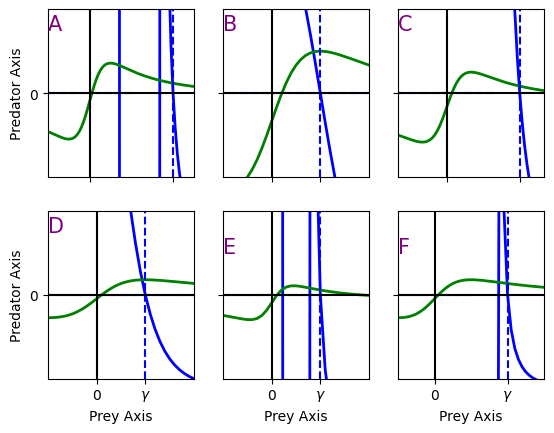}
		\caption{The possible configurations for the prey and predator nullclines of the system (\ref{4mid}) when $0 > \delta \xi - m (1+\alpha \xi) > \frac{-(\delta - m)}{2 \sqrt{\omega}}$.}
		\label{intereqcase24mid}
	\end{minipage}
	
\end{figure}

The intersection of these prey and predator nullclines will result in the following equilibrium points for the system (\ref{4mid}).

\begin{itemize}
	\item Trivial equilibrium $E_0 = (0,0)$. 
	\item Predator free equilibrium $E_1 = (\gamma,0)$.
	\item Pest free equilibrium $E_2 = \left(0,\frac{\delta \xi - m (1+\alpha \xi)}{m \epsilon}\right)$. 
	\item Interior equilibrium $E^* = (x^*,y^*)$.
\end{itemize}

The trivial ($E_0$) and axial equilibria ($E_1$) always exist for the system (\ref{4mid}). Whereas the another axial equilibrium $E_2$ will be in the positive $xy$-quadrant if and only if $\delta \xi - m (1+\alpha \xi) > 0$. Since $\delta > m$, this equilibrium will be in the positive $xy$-quadrant for small $\alpha$ and large $\xi$. In the absence of additional food, this axial equilibrium $E_2 = \left(0,\frac{-1}{\epsilon}\right)$ does not exist in the positive $xy$-quadrant.

The interior equilibrium of the system (\ref{4mid}), if exists, is given by $E^* = (x^*, y^*)$  which is the point of intersection of the non trivial prey nullcline (\ref{4midprey}) and the non trivial predator nullcline (\ref{4midpredator}).

Upon simplification, $E^* = (x^*,y^*)$ is given by
\begin{equation} \label{4midystar}
	y^* = \frac{(\delta - m) x^{*} + \left(\delta \xi - m (1+\alpha \xi) \right) \left( \omega {x^{*}}^2 + 1 \right)}{m \epsilon (\omega {x^{*}}^2 + 1)}.
\end{equation}

and $x^*$ should satisfy the following fifth order equation
\begin{equation} \label{4midxstar}
	\begin{split}
		\left(\frac{\delta \epsilon \xi \omega^2}{\gamma}\right) x^5 	+ \epsilon \omega \left(\frac{\delta}{\gamma} - \xi \omega \left(1 + \delta \right) \right) x^4 + \epsilon \omega \left( \frac{\xi}{\gamma} \left(1 + \delta \right) - \delta \right) x^3 & \\
		+ \left( \frac{\delta \epsilon}{\gamma} + \delta \xi \omega (1 - \epsilon) - m \omega (1 + \alpha \xi)\right) x^2 & \\
		+ \left( \frac{\delta \xi \epsilon }{\gamma} + \delta \left(1 - \epsilon \right) - m \right) x + \left(\delta \xi (1 - \epsilon) - m (1 + \alpha \xi)\right) & = 0.
	\end{split}
\end{equation}

This equation has at most five real roots. However, the Abel–Ruffini theorem states that there is no solution in radicals to general polynomial equations of degree five or higher with arbitrary coefficients. However, the study of nullclines numerically proved the existence of atleast one interior equilibrium point. 

The results obtained so far can be summarized as 

\begin{lem} \label{4midintcond}
	The system (\ref{4mid}) exhibits at least one interior equilibrium point ($E^* = (x^*,y^*)$) which is a solution of the equations (\ref{4midystar}) and (\ref{4midxstar}) and the parameter values satisfying the following constraints
	\begin{enumerate}
		\item $\delta \xi - m (1+\alpha \xi) > \frac{-(\delta - m)}{2 \sqrt{\omega}}$,
		\item $\epsilon (\omega x^2 + 1)\left(1-\frac{x}{\gamma}\right) < 1$.
	\end{enumerate}
\end{lem}

In the absence of mutual interference (i.e., $\epsilon = 0$), $x^*$ is a solution of the quadratic equation $\omega (\delta \xi - m (1 + \alpha \xi)) x^2 + (\delta - m) x + \delta \xi - m (1 + \alpha \xi)  = 0$ and $$y^* = \left(1-\frac{x^*}{\gamma}\right) \left[x^*+(\omega {x^{*}}^2 + 1)(1+\alpha \xi)\right].$$ It also satisfies the condition $\delta \xi - m (1+\alpha \xi) > \frac{-(\delta - m)}{2 \sqrt{\omega}}$.

\section{Stability of Equilibria} \label{sec:4midstab}

In order to obtain the asymptotic behavior of the trajectories of the system (\ref{4mid}), the associated Jacobian matrix is given by 

$$J = \begin{bmatrix}
	\frac{\partial}{\partial x} f(x,y)  & \frac{\partial}{\partial y} f(x,y)\\
	\frac{\partial}{\partial x} g(x,y) & \frac{\partial}{\partial y} g(x,y)
\end{bmatrix},$$
where

\begin{eqnarray*}
	f(x,y) &=& x \left(1-\frac{x}{\gamma} \right)- \frac{xy}{(\omega x^2 + 1)(1+\alpha \xi+\epsilon y)+x}, \\
	g(x,y) &=& \frac{\delta \left( x + \xi (\omega x^2 + 1)\right) y}{(\omega x^2 + 1)(1+\alpha \xi+\epsilon y)+x} - m y,
\end{eqnarray*}
and 
\begin{eqnarray*}
	\frac{\partial}{\partial x} f(x,y) &=& 1 - \frac{2 x}{\gamma} - \frac{y (1 - \omega x^2) (1 + \alpha \xi + \epsilon y)}{\left((\omega x^2 + 1)(1+\alpha \xi+\epsilon y)+x\right)^2}. \\
	\frac{\partial}{\partial y} f(x,y) &=& -\frac{x ((\omega x^2 + 1) (1+ \alpha \xi) + x)}{\left((\omega x^2 + 1)(1+\alpha \xi+\epsilon y)+x\right)^2}. \\
	\frac{\partial}{\partial x} g(x,y) &=& \frac{ \delta  y (1+(\alpha - 1) \xi + \epsilon y) (1 - \omega x^2)}{\left((\omega x^2 + 1)(1+\alpha \xi+\epsilon y)+x\right)^2}. \\
	\frac{\partial}{\partial y} g(x,y) &=& \frac{\delta (x + \xi (\omega x^2 + 1)) ((\omega x^2 + 1) (1 + \alpha \xi) + x)}{\left((\omega x^2 + 1)(1+\alpha \xi+\epsilon y)+x\right)^2} - m.
\end{eqnarray*}

At the trivial equilibrium point $E_0 = (0,0)$, we obtain the jacobian as 

\begin{equation*}
	J\left( E_0 \right) = \begin{bmatrix}
		1  & 0 \\
		0 & \frac{\delta \xi - m (1+ \alpha \xi)}{1+\alpha \xi}
	\end{bmatrix}.
\end{equation*}

The eigenvalues of this jacobian matrix are $1, \frac{\delta \xi - m (1+ \alpha \xi)}{1+\alpha \xi}$. If $\delta \xi - m (1+ \alpha \xi) > 0$, then both the eigen values have same signs. This makes the equilibrium point $E_0 = (0,0)$ unstable. If $\delta \xi - m (1+ \alpha \xi) < 0$, then both the eigen values will have opposite signs. This makes the point $E_0 = (0,0)$ a saddle point. In the absence of additional food, $E_0 = (0,0)$ is a saddle point. 

At the axial equilibrium point $E_1 = (\gamma ,0)$, we obtain the jacobian as 

\begin{equation*}
	J(E_1) = \begin{bmatrix}
		-1  & \frac{-\gamma}{(\omega \gamma^2 + 1) (1+ \alpha \xi)+\gamma} \\
		0 & \frac{(\delta - m) \gamma + (\delta \xi - m (1+\alpha \xi))(\omega \gamma^2 + 1)}{(\omega \gamma^2 + 1) (1+ \alpha \xi)+\gamma}
	\end{bmatrix}.
\end{equation*}

The eigenvalues of this jacobian matrix are $$-1, \ \frac{(\delta - m) \gamma + (\delta \xi - m (1+\alpha \xi))(\omega \gamma^2 + 1)}{(\omega \gamma^2 + 1) (1+ \alpha \xi)+\gamma}.$$

If $(\delta - m) \gamma + (\delta \xi - m (1+\alpha \xi))(\omega \gamma^2 + 1) > 0$, then both the eigenvalues will have opposite sign resulting in a saddle point. If $(\delta - m) \gamma + (\delta \xi - m (1+\alpha \xi))(\omega \gamma^2 + 1) < 0$, then it will be an asymptotically stable node. In the absence of additional food,  $E_1 = (\gamma, 0)$ is a saddle point if $\frac{\delta}{m} > 1 + \frac{\omega \gamma^2 + 1}{\gamma}$. Else it is a stable equilibrium point. 

We now consider another axial equilibrium point which exists only for the additional food provided system (\ref{4mid}). At this axial equilibrium point $E_2 = \left(0, \frac{\delta \xi - m (1+\alpha \xi)}{m \epsilon}\right)$, the associated jacobian matrix is given as 

\begin{equation*}
	J(E_2) = \begin{bmatrix}
		1-\frac{\delta \xi - m (1+\alpha \xi)}{\delta \xi \epsilon}  & 0 \\
		\frac{(\delta - m) (\delta \xi - m (1+\alpha \xi))}{\delta \xi \epsilon} & -\frac{m}{\delta \xi} (\delta \xi - m (1+\alpha \xi))
	\end{bmatrix}.
\end{equation*}

The eigen values for this jacobian matrix are $1-\frac{\delta \xi - m (1+\alpha \xi)}{\delta \xi \epsilon}$ and $-\frac{m}{\delta \xi} (\delta \xi - m (1+\alpha \xi))$. Since this equilibrium point exists in positive $xy$-quadrant only when $\delta \xi - m (1+\alpha \xi) > 0$, the second eigen value of the associated jacobian matrix is always negative. Therefore, the two eigen values will have same negative sign and result in a stable equilibrium point when $\delta \xi - m (1+\alpha \xi) > \delta \xi \epsilon > 0$. If $\delta \xi \epsilon > \delta \xi - m (1+\alpha \xi) > 0$, then eigenvalues are of opposite sign resulting in a saddle equilibrium point. 

The following lemmas present the stability nature of the trivial and axial equilibria. 

\begin{lem}
	The trivial equilibrium $E_0 = (0,0)$ is saddle (unstable node) if 
	\begin{equation*}
		\delta \xi - m (1+\alpha \xi) < (>)\  0.
	\end{equation*}
\end{lem}

\begin{lem}
	The predator-free axial equilibrium $E_1 = (\gamma,0)$ is stable node (saddle) if 
	\begin{equation*}
		\delta \xi - m (1+\alpha \xi) < (>)\  \frac{-(\delta - m) \gamma}{\omega \gamma^2 + 1}.
	\end{equation*}
\end{lem}

\begin{lem}
	The axial equilibrium $E_2 = \left(0, \frac{\delta \xi - m (1+\alpha \xi)}{m \epsilon}\right) $ exists in positive $xy$-quadrant and is stable node (saddle) if 
	\begin{equation*}
		\delta \xi - m (1+\alpha \xi) >\  0 \text{ and }\delta \xi - m (1+\alpha \xi) > (<) \delta \epsilon \xi.
	\end{equation*}
\end{lem}

\subsection{Stability of Interior Equilibrium}

The interior equilibrium point $E^* = (x^*,y^*)$ is the solution of system of equations  (\ref{4midystar}) and (\ref{4midxstar}) and satisfying the conditions in Lemma \ref{4midintcond}.

At this co existing equilibrium point $E^* = (x^*, y^*)$, we obtain the jacobian as 

$$ J(E^*) = \begin{bmatrix}
	\frac{\partial}{\partial x} f(x,y)  & \frac{\partial}{\partial y} f(x,y) \\
	\frac{\partial}{\partial x} g(x,y) & \frac{\partial}{\partial y} g(x,y)
\end{bmatrix} \Bigg|_{(x^*,y^*)} . $$

The associated characteristic equation is given by

\begin{equation}
	\lambda ^2 - \text{Tr } J \bigg|_{(x^*,y^*)} \lambda + \text{Det } J \bigg|_{(x^*,y^*)} = 0.
\end{equation}

Now

\begin{eqnarray*}
	\text{Det } J \bigg|_{(x^*,y^*)} &=& \frac{\partial}{\partial x} f(x,y) \ .\ \frac{\partial}{\partial y} g(x,y) - \frac{\partial}{\partial y} f(x,y) \ .\ \frac{\partial}{\partial x} g(x,y) \\
	&=& \left( 1 - \frac{2 x}{\gamma} - \frac{y (1 - \omega x^2) (1 + \alpha \xi + \epsilon y)}{\left((\omega x^2 + 1)(1+\alpha \xi+\epsilon y)+x\right)^2} \right) \\
	& & \left( \frac{\delta (x + \xi (\omega x^2 + 1)) ((\omega x^2 + 1) (1 + \alpha \xi) + x)}{\left((\omega x^2 + 1)(1+\alpha \xi+\epsilon y)+x\right)^2} - m \right) \\
	& &-  \left( -\frac{x ((\omega x^2 + 1) (1+ \alpha \xi) + x)}{\left((\omega x^2 + 1)(1+\alpha \xi+\epsilon y)+x\right)^2} \right) \\
	& & \left( \frac{ \delta  y (1+(\alpha - 1) \xi + \epsilon y) (1 - \omega x^2)}{\left((\omega x^2 + 1)(1+\alpha \xi+\epsilon y)+x\right)^2} \right)\bigg|_{(x^*,y^*)} .
\end{eqnarray*}

Upon simplification, we have

\begin{eqnarray*}
	\text{Det } J \bigg|_{(x^*,y^*)} &=& \left( 1 - \frac{2 x}{\gamma} \right) \left( \frac{\delta (x + \xi (\omega x^2 + 1)) ((\omega x^2 + 1) (1 + \alpha \xi) + x)}{\left((\omega x^2 + 1)(1+\alpha \xi+\epsilon y)+x\right)^2} - m \right) \\
	& & + \frac{m y (1 - \omega x^2) (1 + \alpha \xi + \epsilon y)}{\left((\omega x^2 + 1)(1+\alpha \xi+\epsilon y)+x\right)^2} \\
	& & - \frac{\delta \xi y (1 - \omega x^2) \left((\omega x^2 + 1)(1+\alpha \xi)+x\right)}{\left((\omega x^2 + 1)(1+\alpha \xi+\epsilon y)+x\right)^3} \bigg|_{(x^*,y^*)} .
\end{eqnarray*}

From (\ref{4mid}), the following equations satisfy at the interior equilibrium point $E^* = (x^*,y^*)$.
$$ 1-\frac{x}{\gamma} = \frac{y}{(\omega x^2 + 1) (1+\alpha \xi+\epsilon y)+x} \bigg|_{(x^*,y^*)}.$$ and $$ \frac{\delta \left(x + \xi (\omega x^2 + 1) \right)}{(\omega x^2 + 1) (1+\alpha \xi+\epsilon y)+x} = m \bigg|_{(x^*,y^*)}.$$

Substituting these two equations in the definition of determinent, we have

\begin{eqnarray*}
	\text{Det } J \bigg|_{(x^*,y^*)} &=& \left( 1 - \frac{2 x}{\gamma} \right) \left( \frac{m ((\omega x^2 + 1) (1 + \alpha \xi) + x)}{(\omega x^2 + 1)(1+\alpha \xi+\epsilon y)+x} - m \right) \\
	& & + \frac{m \left(1 - \frac{x}{\gamma}\right) (1 - \omega x^2) (1 + \alpha \xi + \epsilon y)}{(\omega x^2 + 1)(1+\alpha \xi+\epsilon y)+x} \\
	& & - \frac{\delta \xi \left(1 - \frac{x}{\gamma}\right) (1 - \omega x^2) \left((\omega x^2 + 1)(1+\alpha \xi)+x\right)}{\left((\omega x^2 + 1)(1+\alpha \xi+\epsilon y)+x\right)^2} \bigg|_{(x^*,y^*)} .
\end{eqnarray*}

Upon simplification, we have

\begin{eqnarray*}
	\text{Det } J \bigg|_{(x^*,y^*)} &=&  - \frac{m \epsilon y ((\omega x^2 + 1) \left( 1 - \frac{2 x}{\gamma} \right)}{(\omega x^2 + 1)(1+\alpha \xi+\epsilon y)+x} + \frac{m \left(1 - \frac{x}{\gamma}\right) (1 - \omega x^2) (1 + \alpha \xi + \epsilon y)}{(\omega x^2 + 1)(1+\alpha \xi+\epsilon y)+x} \\
	& & - \frac{\delta \xi \left(1 - \frac{x}{\gamma}\right) (1 - \omega x^2) \left((\omega x^2 + 1)(1+\alpha \xi)+x\right)}{\left((\omega x^2 + 1)(1+\alpha \xi+\epsilon y)+x\right)^2} \bigg|_{(x^*,y^*)} .
\end{eqnarray*}

This further simplifies to 

\begin{eqnarray*}
	\text{Det } J \bigg|_{(x^*,y^*)} &=&  \frac{m \left(1 - \frac{x}{\gamma}\right) (1 - \omega x^2) (1 + \alpha \xi)}{(\omega x^2 + 1)(1+\alpha \xi+\epsilon y)+x} + \frac{m \epsilon x y \left( \frac{3 \omega x^2}{\gamma} - 2 \omega x + \frac{1}{\gamma} \right)}{(\omega x^2 + 1)(1+\alpha \xi+\epsilon y)+x} \\
	& & - \frac{\delta \xi \left(1 - \frac{x}{\gamma}\right) (1 - \omega x^2) \left((\omega x^2 + 1)(1+\alpha \xi)+x\right)}{\left((\omega x^2 + 1)(1+\alpha \xi+\epsilon y)+x\right)^2} \bigg|_{(x^*,y^*)} .
\end{eqnarray*}

From (\ref{4midystar}), we have $$ (\omega {x^*}^2 + 1)(1+\alpha \xi+\epsilon y^*)+x^* = \frac{\delta}{m} \left(x^* + \xi (\omega {x^*}^2 + 1) \right).$$ 

This further simplifies the determinent to 

\begin{eqnarray*}
	\text{Det } J \bigg|_{(x^*,y^*)} &=&  \frac{m \left(1 - \frac{x}{\gamma}\right) (1 - \omega x^2) (1 + \alpha \xi)}{(\omega x^2 + 1)(1+\alpha \xi+\epsilon y)+x} + \frac{m \epsilon x y \left( \frac{3 \omega x^2}{\gamma} - 2 \omega x + \frac{1}{\gamma} \right)}{(\omega x^2 + 1)(1+\alpha \xi+\epsilon y)+x} \\
	& & - \frac{m \xi \left(1 - \frac{x}{\gamma}\right) (1 - \omega x^2)}{\left((\omega x^2 + 1)(1+\alpha \xi+\epsilon y)+x\right)} \ \ . \ \frac{(\omega x^2 + 1)(1+\alpha \xi)+x}{(\omega x^2 + 1)\xi+x} \bigg|_{(x^*,y^*)} .
\end{eqnarray*}

Since $x << \omega x^2 + 1$, we approximate the last term of the above equation and simplify to get the following determinent value

\begin{equation} \label{4midintdet}
	\text{Det } J \bigg|_{E^*} =  \frac{m \epsilon x y \left( \frac{3 \omega x^2}{\gamma} - 2 \omega x + \frac{1}{\gamma} \right)}{(\omega x^2 + 1)(1+\alpha \xi+\epsilon y)+x} \bigg|_{(x^*,y^*)}.
\end{equation}

The sign of the determinent depends on the sign of quadratic expression $h(x) = \frac{3 \omega x^2}{\gamma} - 2 \omega x + \frac{1}{\gamma}$. From the second derivative test, this function has only minimum value at $x = \frac{\gamma}{3}$ and the value is $\frac{3 - \omega \gamma^2}{3 \gamma}$. Also $h(0) = \frac{1}{\gamma} > 0$ and $h(\gamma) = \omega \gamma + \frac{1}{\gamma} > 0$. Therefore, if $h(\frac{\gamma}{3}) > 0$, then the determinent is positive in $(0,\gamma)$. Else, it will be positive in the region except the interval of two positive roots of $h(x)$. The results can be summarized as follows

\begin{lem} \label{4midintdetsign}
	The determinent of the jacobian matrix at the interior equilibrium point is positive when either $\omega \gamma^2 \leq 3$ or $$\omega \gamma^2 > 3 \text{ and } x^* \in \left\{ (0, \gamma) - \left(\frac{\gamma}{3} - \frac{1}{3} \sqrt{\gamma^2 - \frac{3}{\omega}}, \frac{\gamma}{3} + \frac{1}{3} \sqrt{\gamma^2 - \frac{3}{\omega}}\right)\right\}. $$
\end{lem}

In the absence of mutual interference (i.e., $\epsilon = 0$), the determinent is positive when $\left(\delta \xi - m(1+\alpha \xi)\right) \left(1 - \omega {x^*}^2\right) < 0$. Else, the determinent will be negative.

The trace of the jacobian matrix is given by 

\begin{eqnarray*}
	\text{Tr } J \bigg|_{(x^*,y^*)} &=& \frac{\partial}{\partial x} f(x,y) + \frac{\partial}{\partial y} g(x,y) \\
	&=& 1 - \frac{2 x}{\gamma} - \frac{y (1 - \omega x^2) (1 + \alpha \xi + \epsilon y)}{\left((\omega x^2 + 1)(1+\alpha \xi+\epsilon y)+x\right)^2} \\
	& & + \frac{\delta (x + \xi (\omega x^2 + 1)) ((\omega x^2 + 1) (1 + \alpha \xi) + x)}{\left((\omega x^2 + 1)(1+\alpha \xi+\epsilon y)+x\right)^2} - m \bigg|_{(x^*,y^*)}. 
\end{eqnarray*}

From (\ref{4mid}), the following equations satisfy at the interior equilibrium point $E^* = (x^*,y^*)$.
$$ 1-\frac{x}{\gamma} = \frac{y}{(\omega x^2 + 1) (1+\alpha \xi+\epsilon y)+x} \bigg|_{(x^*,y^*)}.$$ and $$ \frac{\delta \left(x + \xi (\omega x^2 + 1) \right)}{(\omega x^2 + 1) (1+\alpha \xi+\epsilon y)+x} = m \bigg|_{(x^*,y^*)}.$$

Substituting these two equations in the trace of jacobian, we have

\begin{eqnarray*}
	\text{Tr } J \bigg|_{(x^*,y^*)} &=& 1 - \frac{2 x}{\gamma} - \frac{ \left(1 - \frac{x}{\gamma}\right) (1 - \omega x^2) (1 + \alpha \xi + \epsilon y)}{(\omega x^2 + 1)(1+\alpha \xi+\epsilon y)+x} \\
	& & + \frac{m ((\omega x^2 + 1) (1 + \alpha \xi) + x)}{(\omega x^2 + 1)(1+\alpha \xi+\epsilon y)+x} - m \bigg|_{(x^*,y^*)}. 
\end{eqnarray*}

Upon simplification, we have

\begin{equation} \label{4midinttrace}
	\begin{split}
		\text{Tr } J \bigg|_{E^*} =& \frac{1}{(\omega x^2 + 1)(1+\alpha \xi+\epsilon y)+x} \Bigg[ - \alpha \xi x \left(\frac{3 \omega x^2}{\gamma} - 2 \omega x + \frac{1}{\gamma}\right)  \\
		& + \left(1 - \frac{2 x}{\gamma}\right) x - x (1 + \epsilon y) \left(\frac{3 \omega x^2}{\gamma} - 2 \omega x + \frac{1}{\gamma}\right) - m \epsilon y (\omega x^2 + 1) \Bigg] \bigg|_{*}. 
	\end{split}
\end{equation}

Here we can consider only the case $\frac{3 \omega x^2}{\gamma} - 2 \omega x + \frac{1}{\gamma} > 0$. Else, from lemma \autoref{4midintdetsign}, the determinent is negative resulting in a saddle point. 

After substituting $1 + \alpha \xi + \epsilon y = \frac{\delta \xi}{m} + \frac{(\delta - m)x}{m(\omega x^2 + 1)}$, the trace of the jacobian is negative when 

\begin{equation} 
	\epsilon >  \frac{m x^* \left(1 - \frac{2 x^*}{\gamma} \right) - \left(\delta \xi + \left( \delta - m\right)\frac{x}{\omega x^2 + 1}\right) \left(\frac{3 \omega {x^*}^2}{\gamma} - 2 \omega x^* + \frac{1}{\gamma} \right)}{m^2 y^* \left(\omega {x^*}^2 + 1\right)^2},
\end{equation}
provided $\frac{3 \omega x^2}{\gamma} - 2 \omega x + \frac{1}{\gamma} > 0$.

The results obtained in this section can be summarized as follows:
\begin{thm}
	The interior equilibrium point ($E^* = (x^*,y^*)$) of the system (\ref{4mid}) exists when it satisfies the conditions in Lemma \ref{4midintcond}. The nature of this equilibrium depends on the signs of determinent and trace of the jacobian matrix given in equations (\ref{4midintdet}) and (\ref{4midinttrace}). The interior equilibrium is 
	\begin{itemize}
		\item a saddle point when $$\frac{3 \omega x^2}{\gamma} - 2 \omega x + \frac{1}{\gamma} < 0 \implies x^* \in \left(\frac{\gamma}{3} - \frac{1}{3} \sqrt{\gamma^2 - \frac{3}{\omega}}, \frac{\gamma}{3} + \frac{1}{3} \sqrt{\gamma^2 - \frac{3}{\omega}}\right).$$
	\end{itemize}		
	When $\frac{3 \omega x^2}{\gamma} - 2 \omega x + \frac{1}{\gamma} > 0$,
	\begin{itemize}
		\item asymptotically stable when $\epsilon >  \frac{m x^* \left(1 - \frac{2 x^*}{\gamma} \right) - \left(\delta \xi + \left( \delta - m\right)\frac{x}{\omega x^2 + 1}\right) \left(\frac{3 \omega {x^*}^2}{\gamma} - 2 \omega x^* + \frac{1}{\gamma} \right)}{m^2 y^* \left(\omega {x^*}^2 + 1\right)^2}$,
		\item an unstable point when $\epsilon <  \frac{m x^* \left(1 - \frac{2 x^*}{\gamma} \right) - \left(\delta \xi + \left( \delta - m\right)\frac{x}{\omega x^2 + 1}\right) \left(\frac{3 \omega {x^*}^2}{\gamma} - 2 \omega x^* + \frac{1}{\gamma} \right)}{m^2 y^* \left(\omega {x^*}^2 + 1\right)^2}$.
	\end{itemize}		
\end{thm}

\section{Bifurcation Analysis} \label{sec:4midbifur}

\subsection{Transcritical Bifurcation}
Within this subsection, we derive the conditions for the existence of transcritical bifurcation near the equilibrium point $E_1 = (\gamma,0)$ using the parameter $\xi$ as the bifurcation parameter.

\begin{thm}
	When the parameter satisfies $(1-\alpha) \gamma + \omega \gamma^2 + 1 \neq 0,\ \delta - m \alpha \neq 0$ and $\xi = \xi^* = \frac{m (\omega \gamma^2 + \gamma + 1)-\delta \gamma}{(\delta - m \alpha)(\omega \gamma^2 + 1)}$, a transcritical bifurcation occurs at $E_1 = (\gamma,0)$ in the system (\ref{4mid}).
\end{thm}

\begin{proof}
	The jacobian matrix corresponding to equilibrium point $E_1=(\gamma,0)$ is given by
	
	\begin{equation*}
		J(E_1) = 
		\begin{bmatrix}
			-1  & \frac{-\gamma}{\gamma + (1 + \alpha \xi) (\omega \gamma^2 + 1)} \\\
			0 & \frac{(\delta - m) \gamma + (\delta \xi - m (1+\alpha \xi))(\omega \gamma^2 + 1)}{\gamma + (1 + \alpha \xi) (\omega \gamma^2 + 1)}
		\end{bmatrix}.
	\end{equation*}
	The eigenvectors corresponding to the zero eigenvalues of $J(E_1)$ and $J(E_1)^{T}$ be denoted by $V$ and $W$, respectively. 
	
	\begin{equation*}
		V = \begin{bmatrix} V_1  \\ V_2 \end{bmatrix} = \begin{bmatrix}	1  \\ - \frac{\gamma + (1 + \alpha \xi) (\omega \gamma^2 + 1)}{\gamma} \end{bmatrix} , \ \ W = \begin{bmatrix} 0 \\ 1 \end{bmatrix}.
	\end{equation*}
	Note that $V_2 < 0$. Let us denote system (\ref{4mid}) as $H = \begin{bmatrix} F  \\ G \end{bmatrix}.$
	Thus, 
	
	\begin{eqnarray*}
		H_{\xi} (E_1; \xi^*) &=&  \begin{bmatrix} 0 \\ 0 \end{bmatrix}, \\
		DH_\xi(E_1; \xi^*)V &=& 
		\begin{bmatrix}
			\frac{\partial F_\xi}{\partial x} & \frac{\partial F_\xi}{\partial y} \\
			\frac{\partial G_\xi}{\partial x} & \frac{\partial G_\xi}{\partial y}
		\end{bmatrix}
		\begin{bmatrix}
			V_1 \\
			V_2
		\end{bmatrix}
		_{(E_1; \xi^*)} = \frac{-\alpha (\omega \gamma^2 + 1)}{\gamma + (\omega \gamma^2 + 1)(1+\alpha \xi)}
		\begin{bmatrix}
			1 \\
			\frac{(1-\alpha)\gamma + \omega \gamma^2 + 1}{\gamma}
		\end{bmatrix}, \\ 
		D^2 H(E_1; \xi^*)(V, V) &=&
		\begin{bmatrix}
			\frac{\partial^2 F}{\partial x^2} V_1^2 + 2 \frac{\partial^2 F}{\partial x \partial y} V_1 V_2 + \frac{\partial^2 F}{\partial y^2} V_2^2 \\
			\frac{\partial^2 G}{\partial x^2} V_1^2 + 2 \frac{\partial^2 G}{\partial x \partial y} V_1 V_2 + \frac{\partial^2 G}{\partial y^2} V_2^2
		\end{bmatrix}
		_{(E_1, \xi^*)} \\
		&=& \begin{bmatrix}
			-\frac{2}{\gamma} - \frac{2 (\omega \gamma^2 - 1) (1 + \alpha \xi)}{\gamma (\gamma + (1 + \alpha \xi)(\omega \gamma^2 + 1))} + \frac{2 \epsilon (\omega \gamma^2 + 1)}{\gamma} \\
			- \frac{2 \delta (1 + \alpha \xi - \xi) ( 1 - \omega \gamma^2)}{\gamma (\gamma + (1 + \alpha \xi) (\omega \gamma^2 + 1))}- \frac{2 \delta \epsilon (\gamma + \xi (\omega \gamma^2 + 1)) (\omega \gamma^2 + 1)}{\gamma^2}
		\end{bmatrix}.
	\end{eqnarray*}
	
	If $1+\omega \gamma^2 + (1-\alpha)\gamma \neq 0$, we have
	
	\begin{eqnarray*}
		W^{T} H_{\xi} (E_1; \xi^*)&=& 0, \\
		W^{T} [DH_\xi(E_1; \xi^*)V]&=& \frac{-\alpha (\omega \gamma^2 + 1) (1+\omega \gamma^2 +(1-\alpha)\gamma)}{\gamma (\gamma + (1 + \alpha \xi) (1+\omega \gamma^2))} \neq 0, \\
		W^{T} [D^2 H(E_1; \xi^*)(V, V) ]&=& \frac{ (1 + \alpha \xi - \xi) ( 1 - \omega \gamma^2)}{\gamma + (1 + \alpha \xi) (\omega \gamma^2 + 1)} + \frac{\epsilon (\gamma + \xi (\omega \gamma^2 + 1)) (\omega \gamma^2 + 1)}{\gamma}\\ & & \neq 0.
	\end{eqnarray*}
	By the Sotomayor's theorem \cite{perko2013differential}, the system (\ref{4mid}) undergoes a transcritical bifurcation around $E_1$ at $\xi = \xi^*$.
\end{proof}

Now, We numerically simulate the bifurcations exhibited by the system (\ref{4mid}) with respect to the bifurcation parameter, the quantity of additional food ($\xi$). We depict the equilibria and their stability for the following set of parameter values, $\gamma = 1.0,\ \alpha = 1.0,\ \epsilon = 0.5,\ \delta = 8.0,\ m=6.0, \ \omega = 4.0$. The two subplots in \autoref{trans14mid} - \autoref{saddle4mid} represents the bifurcation diagrams for the parameter $\xi$ with respect to the prey and predator populations respectively. 

In \autoref{trans14mid}, the nature of interior equilibrium shifts from saddle to stable and the the trivial equilibrium $E_0$ changes its nature around the same point from saddle to unstable node. This results in the transcritical bifurcation at the value $\xi =2.9$. \autoref{trans24mid} represents the transcritical bifurcation where the stability of interior equilibrium ($E^* = (x^*,y^*)$) and the axial eqilibrium ($E_1 = (\gamma,0)$) are exchanged at the bifurcation point $\xi = 2.8$. 

\begin{figure}[!ht]
	\centering
	\includegraphics[width=\textwidth]{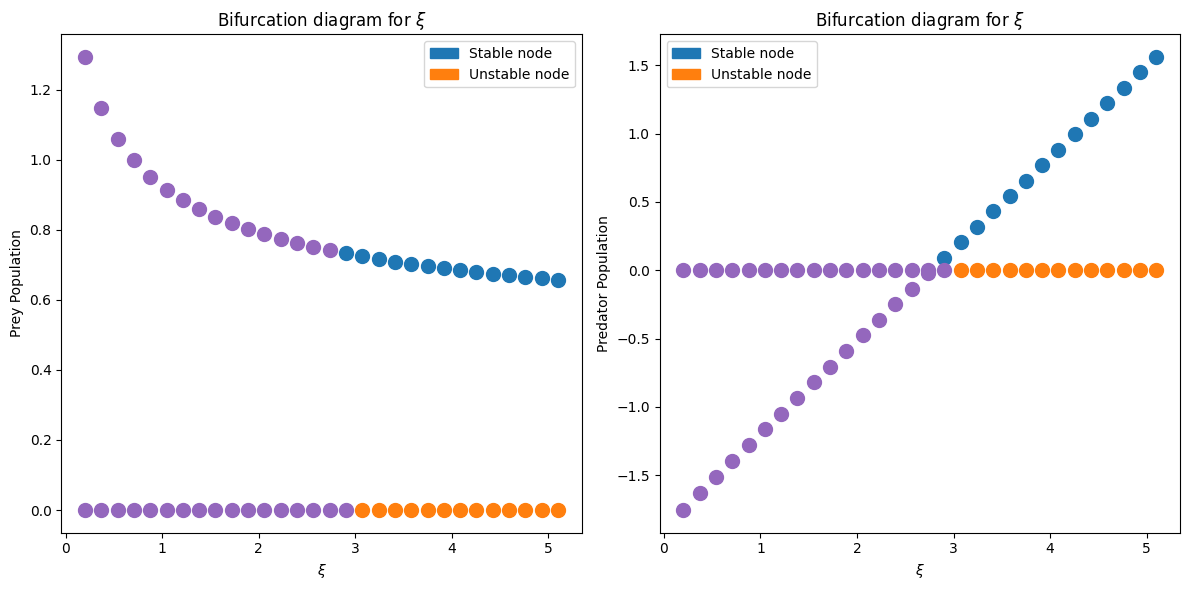}
	\caption{Transcritical bifurcation diagram around trivial equilibrium $E_0 = (0,0)$ with respect to the quantity of additional food $\xi$.}
	\label{trans14mid}
\end{figure}

\begin{figure}[!ht]
	\centering
	\includegraphics[width=\textwidth]{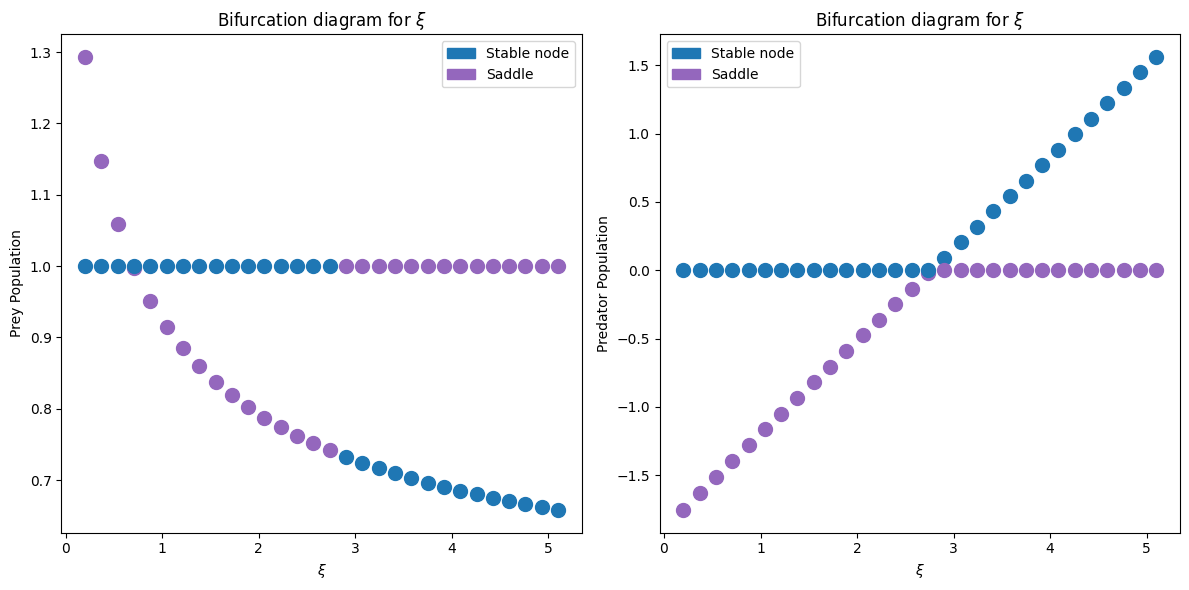}
	\caption{Transcritical bifurcation diagram around axial equilibrium $E_1 = (\gamma,0)$ with respect to the quantity of additional food $\xi$.}
	\label{trans24mid}
\end{figure}

\subsection{Saddle-node Bifurcation}
Within this subsection, we derive the conditions for the existence of saddle-node bifurcation near the equilibrium point $E_2 = \left( 0, \frac{\delta \xi - m (1+\alpha \xi)}{m \epsilon} \right)$ using the parameter $\xi$ as the bifurcation parameter.

\begin{thm}
	When the parameter satisfies $\delta \neq m \alpha,\ \epsilon \delta \xi \neq \delta \xi - m (1 + \alpha \xi)$ and $\xi = \xi^* = \frac{m}{\delta - m \alpha}$, a saddle-node bifurcation occurs at $E_2 = \left( 0, \frac{\delta \xi - m (1+\alpha \xi)}{m \epsilon} \right)$ in the system (\ref{4mid}).
\end{thm}

\begin{proof}
	The jacobian matrix corresponding to equilibrium point $E_2 = \left( 0, \frac{\delta \xi - m (1+\alpha \xi)}{m \epsilon} \right)$ is given by
	
	\begin{equation*}
		J(E_2) =
		\begin{bmatrix}
			1 - \frac{\delta \xi - m (1 + \alpha \xi)}{\delta \xi \epsilon} & 0 \\
			\frac{(\delta \xi - m (1 + \alpha \xi))(\delta - m)}{\epsilon \delta \xi} & -\frac{m \left(\delta \xi - m (1+\alpha \xi)\right)}{\delta \xi}
		\end{bmatrix}.
	\end{equation*}
	
	The eigenvectors corresponding to the zero eigenvalues of $J(E_2)$ and $J(E_2)^{T}$ be denoted by $V$ and $W$, respectively. 
	
	\begin{equation*}
		V = \begin{bmatrix} V_1  \\ V_2 \end{bmatrix} = \begin{bmatrix}	0  \\ 1 \end{bmatrix}, \ \ W = \begin{bmatrix} 1 \\ \frac{1}{\delta - m} - \frac{\epsilon \delta \xi}{(\delta - m)(\delta \xi - m (1 + \alpha \xi))} \end{bmatrix}.
	\end{equation*}
	
	Let us denote system (\ref{4mid}) as $H = \left[ \begin{matrix} F  \\ G \end{matrix} \right].$
	Thus, 
	
	\begin{eqnarray*}
		H_{\xi} (E_2; \xi^*) &=&  \begin{bmatrix} 0 \\ \frac{\delta \xi - m (1+\alpha\xi)}{\epsilon \delta \xi} (\delta - m \alpha) \end{bmatrix}, \\
		DH_\xi(E_2; \xi^*)V &=& 
		\begin{bmatrix}
			\frac{\partial F_\xi}{\partial x} & \frac{\partial F_\xi}{\partial y} \\
			\frac{\partial G_\xi}{\partial x} & \frac{\partial G_\xi}{\partial y}
		\end{bmatrix}
		\begin{bmatrix}
			V_1 \\
			V_2
		\end{bmatrix}
		_{(E_2; \xi^*)} \\
		&=& 
		\begin{bmatrix}	0 \\
			\frac{m^2}{\delta^2 \xi^2} \left[ \delta \alpha \xi + (1 + \alpha \xi) (\delta - 2 m \alpha)\right]
		\end{bmatrix}, \\
		D^2 H(E_2; \xi^*)(V, V) &=&
		\begin{bmatrix}
			\frac{\partial^2 F}{\partial x^2} V_1^2 + 2 \frac{\partial^2 F}{\partial x \partial y} V_1 V_2 + \frac{\partial^2 F}{\partial y^2} V_2^2 \\
			\frac{\partial^2 G}{\partial x^2} V_1^2 + 2 \frac{\partial^2 G}{\partial x \partial y} V_1 V_2 + \frac{\partial^2 G}{\partial y^2} V_2^2
		\end{bmatrix}_{(E_2, \xi^*)} = \begin{bmatrix} 0 \\ - \frac{2 m^3 \epsilon (1 + \alpha \xi)}{\delta^2 \xi^2}	\end{bmatrix}.
	\end{eqnarray*}
	
	If $\delta \neq m \alpha$, we have
	
	\begin{eqnarray*}
		W^{T} H_{\xi} (E_2; \xi^*) &=& \frac{\left((1 - \epsilon) \delta \xi - m (1+\alpha\xi)\right) (\delta - m \alpha)}{\epsilon \delta \xi (\delta - m)} \neq 0, \\
		W^{T} [DH_\xi(E_2; \xi^*)V] &=& \frac{m^2 \left( \delta \alpha \xi + (1 + \alpha \xi) (\delta - 2 m \alpha )\right)}{\delta^2 \xi^2 (\delta - m)} \left(1 - \frac{\epsilon \delta \xi }{\delta \xi - m (1 + \alpha \xi)}\right)\neq 0, \\
		W^{T} [D^2 H(E_2; \xi^*)(V, V) ] &=& - \frac{2 m^3 \epsilon (1 + \alpha \xi)}{\delta^2 \xi^2 (\delta - m)}  \left(1 - \frac{\epsilon \delta \xi }{\delta \xi - m (1 + \alpha \xi)}\right) \neq 0.
	\end{eqnarray*}
	By the Sotomayor's theorem \cite{perko2013differential}, the system (\ref{4mid}) undergoes a saddle-node bifurcation around $E_2$ at $\xi = \xi^*$.
\end{proof}

In \autoref{saddle4mid}, the saddle-node bifurcation around the another axial equilibrium $E_2$ with respect to $\xi$ is discussed. For the parameter values $\gamma = 1.0,\ \alpha = 1.0,\ \epsilon = 0.5,\ \delta = 8.0,\ m=6.0, \ \omega = 4.0$, both the equilibria $E_0$ and $E_2$ exist and move towards each other as $\xi$ reduces. Further at $\xi = 3.0$, both ($E_0$ and $E_2$) collide and become $E_0$ which ensures the happening of saddle-node bifurcation. When $\xi < 3.0$, then there does not exist any prey-free equilibrium $E_2$.

\begin{figure}[!ht]
	\centering
	\includegraphics[width=\textwidth]{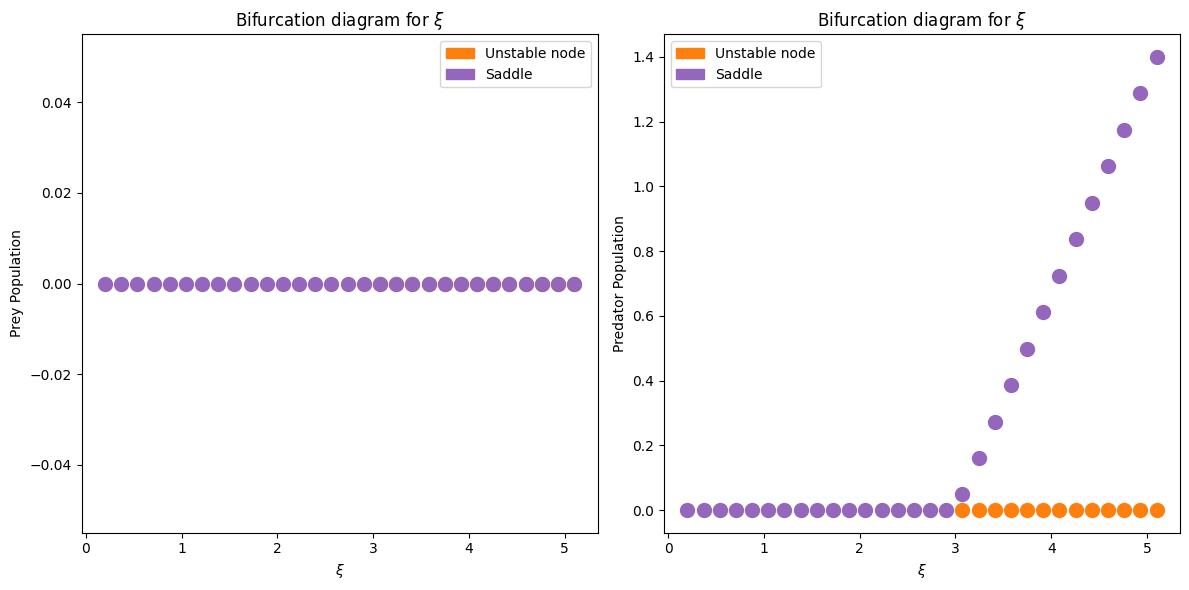}
	\caption{Saddle-node bifurcation diagram around axial equilibrium $E_2 = \left(0,\frac{\delta \xi - m (1 + \alpha \xi)}{m \epsilon}\right)$ with respect to the quantity of additional food $\xi$.}
	\label{saddle4mid}
\end{figure}

\subsection{Hopf Bifurcation}

Within this subsection, we derive the conditions for the existence of Hopf bifurcation near the equilibrium point $E^* = \left( x^*, y^* \right)$ using the parameter $\xi$ as the bifurcation parameter.

\begin{thm}
	Let the following conditions be satisfied by the parameters of system (\ref{4mid}), 
	\begin{itemize}
		\item  $\xi = \xi^* = \frac{ \left(1 - \frac{2 x}{\gamma}\right) x - m \epsilon y (\omega x^2 + 1)}{\alpha x \left(\frac{3 \omega x^2}{\gamma} - 2 \omega x + \frac{1}{\gamma}\right) }\bigg|_{*} - \frac{1 + \epsilon y^*}{\alpha},$
		\item  $\frac{3 \omega}{\gamma}{x^*}^2 - 2 \omega x^* + \frac{1}{\gamma} > 0,$
		\item $x^* \neq \frac{\gamma}{3} \pm \frac{1}{3} \sqrt{\gamma^2 - \frac{3}{\omega}}$ at $\xi = \xi^*$,
	\end{itemize}
	then the system (\ref{4mid}) experiences Hopf bifurcation with respect to quantity of additional food $\xi$ about the interior equilibrium point $E^* = (x^*,y^*)$. Here $\xi^*$ is the critical quantity of additional food at which the Hopf bifurcation occurs.
\end{thm}

\begin{proof}
	The characteristic equation of Jacobian matrix $J_{E^*}$ is given by 
	
	\begin{equation}
		\lambda^2 - \text{Tr}(J_{E^*}) \lambda + \text{Det}(J_{E^*}) = 0,
	\end{equation}
	where the expression of $\text{Tr}(J_{E^*})$ and $\text{Det}(J_{E^*})$ are given by 
	
	From (\ref{4midinttrace}), it is obvious that $\text{Tr}(J_{E^*}) = 0$ when 
	\begin{equation} 
		\begin{split}
			\xi = \xi^* = \frac{ \left(1 - \frac{2 x}{\gamma}\right) x - m \epsilon y (\omega x^2 + 1)}{\alpha x \left(\frac{3 \omega x^2}{\gamma} - 2 \omega x + \frac{1}{\gamma}\right) } - \frac{1 + \epsilon y}{\alpha} \bigg|_{*}. 
		\end{split}
	\end{equation}
	
	Since both eigenvalues at Hopf bifurcation point are purely imaginary numbers, we have Det$(J_{E^*}) > 0$. From lemma \autoref{4midintdetsign}, the determinent is positive when $\frac{3 \omega x^2}{\gamma} - 2 \omega x + \frac{1}{\gamma} > 0$. This happens when either $\omega \gamma^2 \leq 3$ or $$\omega \gamma^2 > 3 \text{ and } x^* \in \left\{ (0, \gamma) - \left(\frac{\gamma}{3} - \frac{1}{3} \sqrt{\gamma^2 - \frac{3}{\omega}}, \frac{\gamma}{3} + \frac{1}{3} \sqrt{\gamma^2 - \frac{3}{\omega}}\right)\right\}. $$
	
	We also have $\frac{d\left(\text{Tr}(J_{E^*}) \right)}{d \xi} \bigg|_{\xi^*} = \frac{3 \omega}{\gamma} {x^*}^2 - 2 \omega x^* + \frac{1}{\gamma} \neq 0.$ Therefore, $x^* \neq \frac{\gamma}{3} \pm \frac{1}{3} \sqrt{\gamma^2 - \frac{3}{\omega}}$.
	
	Therefore, when these conditions are satisfied, the Implicit theorem (Ref) guarentees the occurrence of Hopf bifurcation at $E^* (x^*,y^*)$ i.e., small amplitude periodic solutions bifurcate from $E^* (x^*,y^*)$ through a Hopf bifurcation.
	
\end{proof}

The existence of Hopf bifurcation with respect to $\epsilon$ is depicted in \autoref{hopf4mid}. Stable limit cycle is observed around the interior equilibrium when $\epsilon = 0.35$ and the limit cycle disappears when $\epsilon$ is increased to $0.4$. The remaining parameter values are as follows: $\gamma = 10.0,\ \alpha = 0.1,\ \xi = 0.45,\ \delta = 0.45,\ m=0.28, \ \omega = 0.01$.

\begin{figure}[ht]
	\centering
	\includegraphics[width=\textwidth]{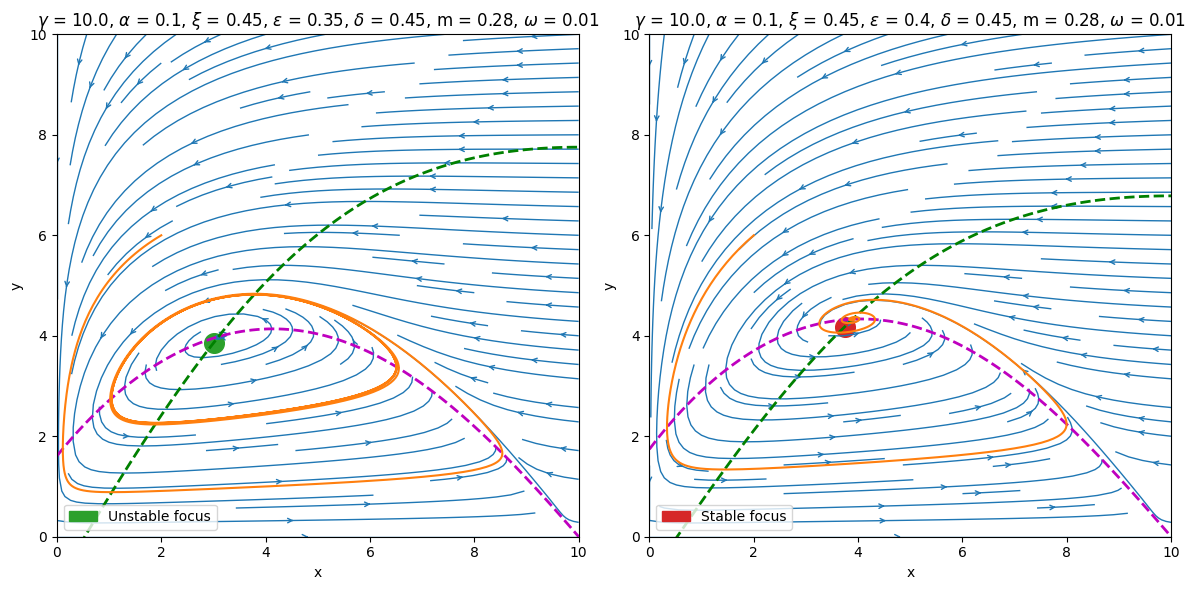}
	\caption{Supercritical Hopf bifurcation diagram with respect to the mutual interference $\epsilon$.}
	\label{hopf4mid}
\end{figure}

\section{Global Dynamics}
\label{sec:4midglobaldynamics}

In this section, we study the global dynamics of the system (\ref{4mid}) in the $\alpha$-$\xi$ parameter space. For this study, we divide the parameter space of the system (\ref{4mid}) in the absence of additional food into three regions. These regions are divided based on the qualitative behaviors of the interior equilibria of the system. They are divided into three regions, namely, $R_1,\ R_2,\ R_3$ corresponding to the space where there is no interior equilibria, one stable interior equilibria and two stable interior equilibria respectively. \autoref{initial4mid} depicts the phase portrait of the initial system in all three regions. 

\begin{figure}[!ht]
	\centering
	\includegraphics[width=\linewidth]{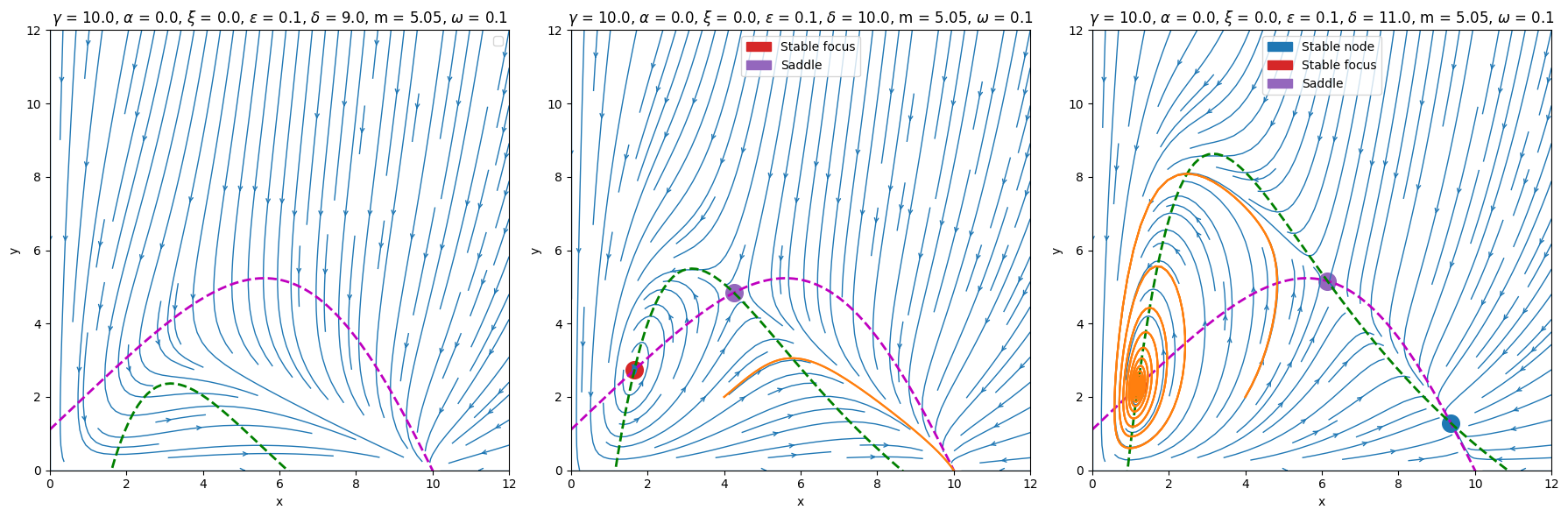}  
	\caption{Dynamics of the system (\ref{4mid}) in the absence of additional food.}
	\label{initial4mid}
\end{figure}

Now, in each of these regions, we study the influence of additional food by dividing the $\alpha - \xi$ parameter space into regions based on the following curves. Each of these curves divide the space into two regions based on the qualitative nature of the equilibrium point. 

\begin{equation*}
	\begin{split}
		\text{Bifurcation Curve for } E_0: &\ \phi_1(\alpha,\xi) : \  \delta \xi - m (1 + \alpha \xi) = 0. \\
		\text{Bifurcation Curve for } E_1: &\ \phi_2(\alpha,\xi) : \ \delta \xi - m (1 + \alpha \xi) + \frac{(\delta - m) \gamma}{\omega \gamma^2 + 1} = 0. \\
		\text{Bifurcation Curve for } E_2: &\ \phi_3(\alpha,\xi) : \ \delta \xi - m (1 + \alpha \xi) - \delta \epsilon \xi = 0. \\
		\text{Existence Curve for } E^*: &\ \phi_4(\alpha,\xi) : \ \delta \xi - m (1 + \alpha \xi) + \frac{(\delta - m)}{2 \sqrt{\omega}} = 0. \\			
	\end{split}
\end{equation*} 

\begin{figure}[!ht]
	\centering
	\includegraphics[width=0.7\linewidth]{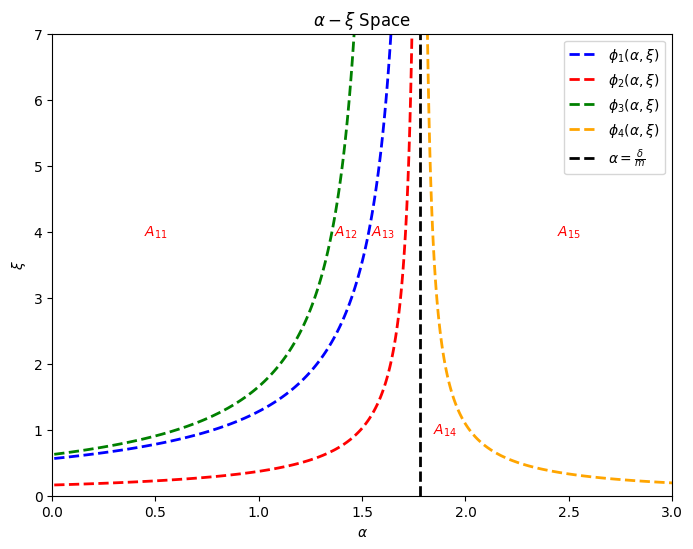} \\
	\caption{Influence of additional food on the system (\ref{4mid}) when the parameters belong to the region $R_1$.}
	\label{r14mid}
\end{figure}

\autoref{r14mid} divides the $\alpha - \xi$ space into $5$ regions. In this region, there is no interior equilibrium in the absence of additional food. As additional food is provided, the prey-free equilibrium $E_2$ is stable in the region $A_{11}$. Also, the predator-free equilibrium $E_1$ is stable in regions $A_{14}$ and $A_{15}$. Provision of additional food leads to the emergence of interior equilibria $E^*$ in the regions $A_{11},\ A_{12},\ A_{13},\ A_{14}$. However, provision of additional food in the region $A_{15}$ takes us to the original qualitative behavior as in the case of absence of additional food.

\begin{figure}[!ht]
	\centering
	\includegraphics[width=0.7\linewidth]{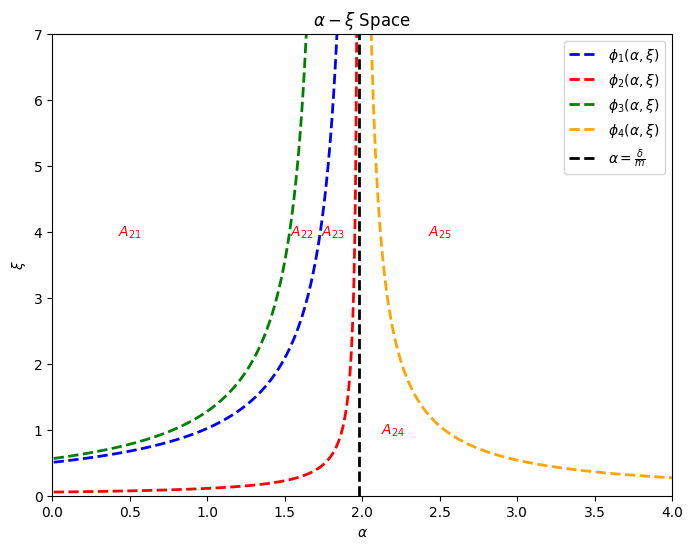} \\
	\caption{Influence of additional food on the system (\ref{4mid}) when the parameters belong to the region $R_2$.}
	\label{r24mid}
\end{figure}

\autoref{r24mid} divides the $\alpha - \xi$ space into $5$ regions. In this region, there is only one stable interior equilibrium in the absence of additional food. As additional food is provided, the prey-free equilibrium $E_2$ is also stable in the region $A_{21}$. Also, the predator-free equilibrium $E_1$ is stable in regions $A_{24}$ and $A_{25}$. As additional food is provided, the system (\ref{4mid}) continues to exhibit interior equilibria $E^*$ in the regions $A_{21},\ A_{22},\ A_{23}$ and $A_{24}$. However, provision of additional food in the region $A_{25}$ leads to the disappearance of interior equilibria.

\begin{figure}[!ht]
	\centering
	\includegraphics[width=0.7\linewidth]{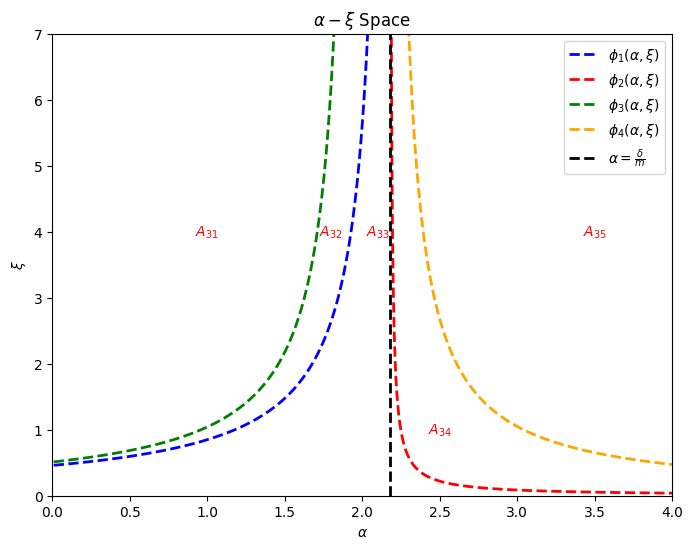} \\
	\caption{Influence of additional food on the system (\ref{4mid}) when the parameters belong to the region $R_3$.}
	\label{r34mid}
\end{figure}

\autoref{r34mid} divides the $\alpha - \xi$ space into $5$ regions. In this region, there are two stable interior equilibria in the absence of additional food. As additional food is provided, the prey-free equilibrium $E_2$ is stable in the region $A_{31}$. Also, the predator-free equilibrium $E_1$ is stable in regions $A_{34}$ and $A_{35}$. As additional food is provided, the system (\ref{4mid}) continues to exhibit interior equilibria $E^*$ in the regions $A_{31},\ A_{32},\ A_{33}$ and $A_{34}$. However, provision of additional food in the region $A_{35}$ leads to the disappearance of interior equilibria.

\section{Consequences of providing Additional Food}\label{sec:4midconseq}

\indent In this section, we explore how the introduction of additional food influences the qualitative stability of equilibria within the system. First, when multiple stable equilibria are present, the ultimate trajectory of the system is determined by the initial densities of both prey and predator populations. Second, when the system evolves toward an axial equilibrium, it has significant implications for biological conservation strategies and pest control efforts. Accordingly, we analyze how the provision of supplementary food can result in the emergence of stable axial equilibria and bistability phenomena.

Regardless of the system’s behavior in the absence of additional food, in regions $A_{12},\ A{13},\ A_{22},\ A_{23},\ A_{32}$ and $A_{33}$, the interior equilibrium remains the sole stable state when additional food is supplied. As a result, trajectories do not approach axial equilibria in these regions. Even when bistability arises, it is confined to multiple interior equilibria rather than involving any axial equilibria.

In contrast, in regions $A_{15},\ A_{25}$ and $A_{35}$, the only stable state achieved is the predator-free axial equilibrium, again independent of the system’s dynamics without additional food. Thus, adding food in these regions directs the system exclusively toward axial equilibria, precluding any occurrence of bistability.

In regions $A_{14},\ A_{24}$ and $A_{34}$, however, both a stable predator-free axial equilibrium and a stable interior equilibrium coexist. This coexistence allows for the possibility of bistability, with the system’s long-term outcome depending sensitively on the initial conditions. A similar situation arises when considering prey-free axial equilibria in the regions $A_{11},\ A_{21}$ and $A_{31}$, where the presence of an interior equilibrium and a stable prey-free state introduces the potential for multiple stable outcomes.

Thus, based on whether the goal is prey conservation or prey eradication, careful selection of initial population sizes becomes crucial to steering the system toward the desired equilibrium state.

\autoref{equiplot4mid} numerically depicts the stability nature of various equilibria that the system exhibits only when additional food and competition terms are altered. This shows the importance of these two terms in the dynamics of the system (\ref{4mid}). Each frame depicts the existence of $0-3$ interior equilibrium.

\begin{figure}[!ht]
	\includegraphics[width=\textwidth]{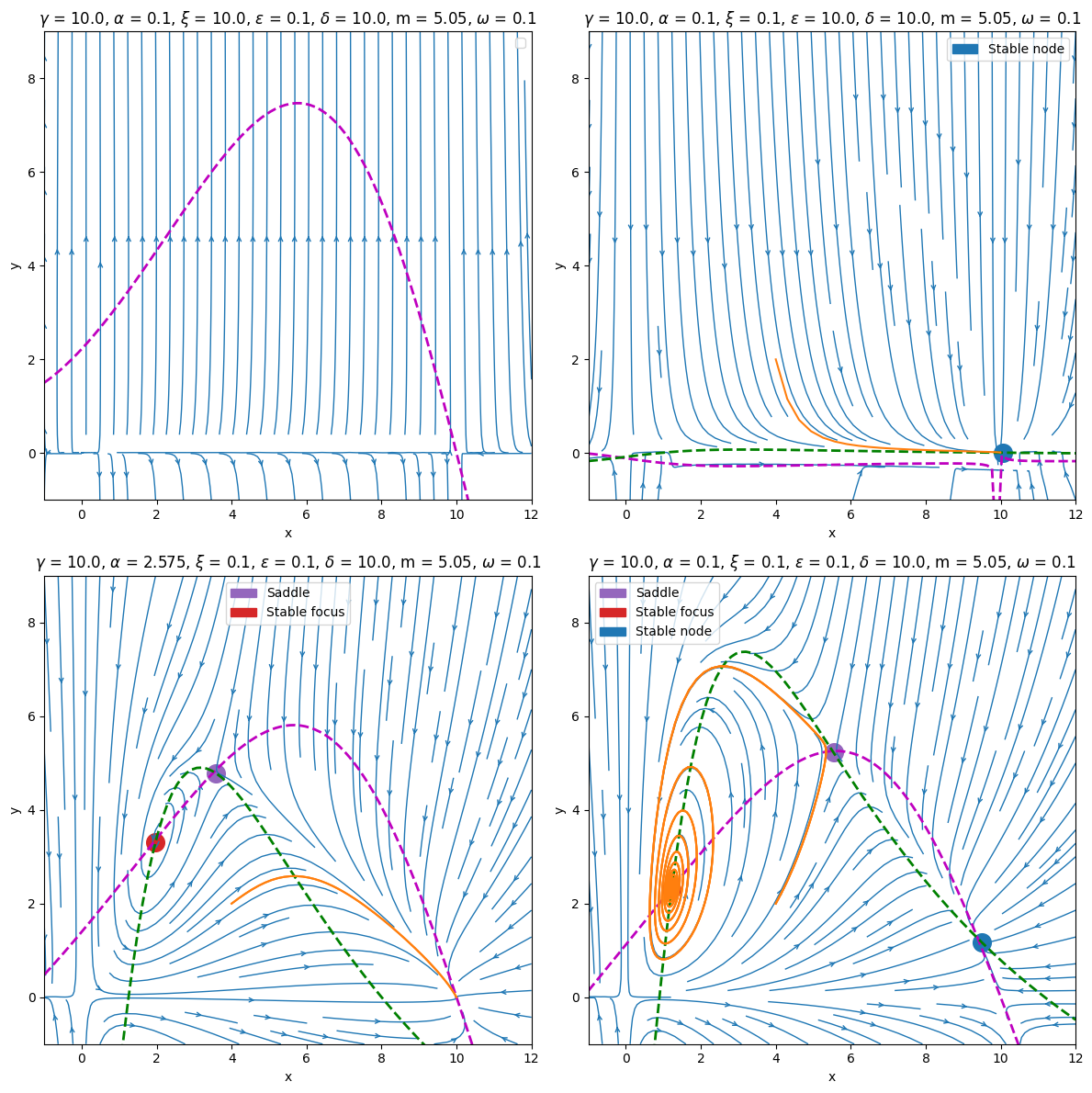}
	\caption{The stability nature of various equilibria of the system (\ref{4mid}).}
	\label{equiplot4mid}
\end{figure}

\section{Time-Optimal Control Studies} \label{sec:4midtimecontrol}

In this section, we formulate and characterise two time-optimal control problems with quality of additional food and quantity of additional food as control parameters respectively. We shall drive the system (\ref{4mid}) from the initial state $(x_0,y_0)$ to the final state $(\bar{x},\bar{y})$ in minimum time.

\subsection{Quality of Additional Food as Control Parameter}

We assume that the quantity of additional food $(\xi)$ is constant and the quality of additional food varies in $[\alpha_{\text{min}},\alpha_{\text{max}}]$. The time-optimal control problem with additional food provided prey-predator system involving Holling type-IV functional response among mutually interfering predators (\ref{4mid}) with quality of additional food ($\alpha$) as control parameter is given by

\begin{equation}
	\begin{rcases}
		& \displaystyle {\bf{\min_{\alpha_{\min} \leq \alpha(t) \leq \alpha_{\max}} T}} \\
		& \text{subject to:} \\
		& \frac{\mathrm{d} x}{\mathrm{d} t} = x \left(1-\frac{x}{\gamma} \right)- \frac{xy}{(1+\alpha \xi + \epsilon y)(\omega x^2 + 1) + x}, \\
		& \frac{\mathrm{d} y}{\mathrm{d} t} = \delta \left( \frac{x + \xi (\omega x^2 + 1)}{(1+\alpha \xi + \epsilon y)(\omega x^2 + 1) + x} \right) y - m y, \\
		& (x(0),y(0)) = (x_0,y_0) \ \text{and} \ (x(T),y(T)) = (\bar{x},\bar{y}).
	\end{rcases}
	\label{4midalpha0}
\end{equation}

This problem can be solved using a transformation on the independent variable $t$ by introducing an independent variable $s$ such that $\mathrm{d}t = ((1+\alpha \xi + \epsilon y)(\omega x^2 + 1) + x) \ \mathrm{d}s$. This transformation converts the time-optimal control problem ($\ref{4midalpha0}$) into the following linear problem.

\begin{equation}
	\begin{rcases}
		& \displaystyle {\bf{\min_{\alpha_{\min} \leq \alpha(t) \leq \alpha_{\max}} S}} \\
		& \text{subject to:} \\
		& \mathring{x}(s) = x \left(1 - \frac{x}{\gamma}\right) ((1+\alpha \xi + \epsilon y)(\omega x^2 + 1) + x) - x y, \\
		& \mathring{y}(s) = \delta (x + \xi (\omega x^2 + 1)) y - ((1+\alpha \xi + \epsilon y)(\omega x^2 + 1) + x) m y, \\
		& (x(0),y(0)) = (x_0,y_0) \ \text{and} \ (x(S),y(S)) = (\bar{x},\bar{y}).
	\end{rcases}
	\label{4midalpha}
\end{equation}

We now calculate the Hamiltonian function which encapsulates the system’s dynamics and optimality conditions \cite{liberzon2011calculus}. Hamiltonian function for this problem (\ref{4midalpha}) is given by
\begin{equation*}
	\begin{split}
		\mathbb{H}(s,x,y,p,q) =& p \left[x \left(1 - \frac{x}{\gamma}\right) ((1+\alpha \xi + \epsilon y)(\omega x^2 + 1) + x) - x y \right] \\ & + q \left[\delta (x + \xi (\omega x^2 + 1)) y - ((1+\alpha \xi + \epsilon y)(\omega x^2 + 1) + x) m y \right] \\ 
		=& \left[ p x \xi  \left(1 - \frac{x}{\gamma}\right) (\omega x^2 + 1) - q \xi m y (\omega x^2 + 1) \right] \alpha \\ &  + p \left[x \left(1 - \frac{x}{\gamma}\right) ((1+ \epsilon y)(\omega x^2 + 1) + x) - x y \right] \\ 
		& + q \left[\delta (x + \xi (\omega x^2 + 1)) y - ((1+ \epsilon y)(\omega x^2 + 1) + x) m y \right]. \\
	\end{split}
\end{equation*}

Here, $p$ and $q$ are costate variables satisfying the adjoint equations 

\begin{equation}
	\begin{split}
		\mathring{p} = \frac{\mathrm{d} p}{\mathrm{d} s} = -  \frac{\partial \mathbb{H}}{\partial x} =  & p \left[y - x \left(2- \frac{3x}{\gamma}\right) - (1+\alpha \xi + \epsilon y) \left(1 - \frac{2 x}{\gamma} + 3 \omega x^2 - \frac{4 \omega x^3}{\gamma}\right)\right] \\
		& + q y \left[m \left(2 \omega x (1+\alpha\xi+\epsilon y)+1\right) - \delta (1+2 \xi \omega x)\right], \\
		\mathring{q} = \frac{\mathrm{d} q}{\mathrm{d} s} = -  \frac{\partial \mathbb{H}}{\partial y} =  & p x \left[1 - \epsilon \left(\omega x^2 + 1 \right)\left(1 - \frac{x}{\gamma} \right) \right] \\
		& + q \left[ (m - \delta) x + (\omega x^2 + 1) (m (1+\alpha \xi + 2 \epsilon y) - \delta \xi)  \right].
	\end{split}
\end{equation}

Since Hamiltonian is a linear function in $\alpha$, the existence of optimal control is guarenteed by \textit{Filippov existence theorem} \cite{cesari2012optimization}.  Since we are minimizing the Hamiltonian, the optimal strategy will be a combination of bang-bang and singular controls which is given by 

\begin{equation}
	\alpha^*(t) =
	\begin{cases}
		\alpha_{\max}, &\text{ if } \frac{\partial \mathbb{H}}{\partial \alpha} < 0.\\
		\alpha_{\min}, &\text{ if } \frac{\partial \mathbb{H}}{\partial \alpha} > 0.
	\end{cases}
\end{equation}
where
\begin{equation}
	\frac{\partial \mathbb{H}}{\partial \alpha} = p x \xi  \left(1 - \frac{x}{\gamma}\right) (\omega x^2 + 1) - q \xi m y (\omega x^2 + 1).
\end{equation}

This problem (\ref{4midalpha}) admits a singular solution if there exists an interval $[s_1,s_2]$ on which $\frac{\partial \mathbb{H}}{\partial \alpha} = 0$. Therefore, the solution is a combination of bang-bang and singular controls.

\subsection{Quantity of Additional Food as Control Parameter}

We assume that the quality of additional food $(\alpha)$ is constant and the quantity of additional food varies in $[\xi_{\text{min}},\xi_{\text{max}}]$. The time-optimal control problem with additional food provided prey-predator system involving Holling type-IV functional response among mutually interfering predators (\ref{4mid}) with quantity of additional food ($\xi$) as control parameter is given by

\begin{equation}
	\begin{rcases}
		& \displaystyle {\bf{\min_{\xi_{\min} \leq \xi(t) \leq \xi_{\max}} T}} \\
		& \text{subject to:} \\
		& \frac{\mathrm{d} x}{\mathrm{d} t} = x \left(1-\frac{x}{\gamma} \right)- \frac{xy}{(1+\alpha \xi + \epsilon y)(\omega x^2 + 1) + x}, \\
		& \frac{\mathrm{d} y}{\mathrm{d} t} = \delta \left( \frac{x + \xi (\omega x^2 + 1)}{(1+\alpha \xi + \epsilon y)(\omega x^2 + 1) + x} \right) y - m y, \\
		& (x(0),y(0)) = (x_0,y_0) \ \text{and} \ (x(T),y(T)) = (\bar{x},\bar{y}).
	\end{rcases}
	\label{4midxi0}
\end{equation}

This problem can be solved using a transformation on the independent variable $t$ by introducing an independent variable $s$ such that $\mathrm{d}t = ((1+\alpha \xi + \epsilon y)(\omega x^2 + 1) + x) \ \mathrm{d}s$. This transformation converts the time-optimal control problem ($\ref{4midxi0}$) into the following linear problem.

\begin{equation}
	\begin{rcases}
		& \displaystyle {\bf{\min_{\xi_{\min} \leq \xi(t) \leq \xi_{\max}} S}} \\
		& \text{subject to:} \\
		& \frac{\mathrm{d} x}{\mathrm{d} t} = x \left(1-\frac{x}{\gamma} \right)- \frac{xy}{(1+\alpha \xi + \epsilon y)(\omega x^2 + 1) + x}, \\
		& \frac{\mathrm{d} y}{\mathrm{d} t} = \delta \left( \frac{x + \xi (\omega x^2 + 1)}{(1+\alpha \xi + \epsilon y)(\omega x^2 + 1) + x} \right) y - m y, \\
		& (x(0),y(0)) = (x_0,y_0) \ \text{and} \ (x(S),y(S)) = (\bar{x},\bar{y}).
	\end{rcases}
	\label{4midxi}
\end{equation}

We now calculate the Hamiltonian function which encapsulates the system’s dynamics and optimality conditions \cite{liberzon2011calculus}. Hamiltonian function for this problem (\ref{4midxi}) is given by
\begin{equation*}
	\begin{split}
		\mathbb{H}(s,x,y,p,q) =& p \left[x \left(1 - \frac{x}{\gamma}\right) ((1+\alpha \xi + \epsilon y)(\omega x^2 + 1) + x) - x y \right] \\ & + q \left[\delta (x + \xi (\omega x^2 + 1)) y - ((1+\alpha \xi + \epsilon y)(\omega x^2 + 1) + x) m y \right] \\ 
		=& \left[ p x \alpha \left(1 - \frac{x}{\gamma}\right) (\omega x^2 + 1) + q \delta y (\omega x^2 + 1) - q \alpha m y (\omega x^2 + 1) \right] \xi \\ &  + p \left[x \left(1 - \frac{x}{\gamma}\right) ((1+ \epsilon y)(\omega x^2 + 1) + x) - x y \right] \\ 
		& + q \left[\delta x y - ((1+ \epsilon y)(\omega x^2 + 1) + x) m y \right]. \\
	\end{split}
\end{equation*}

Here, $p$ and $q$ are costate variables satisfying the adjoint equations 

\begin{equation}
	\begin{split}
		\mathring{p} = \frac{\mathrm{d} p}{\mathrm{d} s} = -  \frac{\partial \mathbb{H}}{\partial x} =  & p \left[y - x \left(2- \frac{3x}{\gamma}\right) - (1+\alpha \xi + \epsilon y) \left(1 - \frac{2 x}{\gamma} + 3 \omega x^2 - \frac{4 \omega x^3}{\gamma}\right)\right] \\
		& + q y \left[m \left(2 \omega x (1+\alpha\xi+\epsilon y)+1\right) - \delta (1+2 \xi \omega x)\right], \\
		\mathring{q} = \frac{\mathrm{d} q}{\mathrm{d} s} = -  \frac{\partial \mathbb{H}}{\partial y} =  & p x \left[1 - \epsilon \left(\omega x^2 + 1 \right)\left(1 - \frac{x}{\gamma} \right) \right] \\
		& + q \left[ (m - \delta) x + (\omega x^2 + 1) (m (1+\alpha \xi + 2 \epsilon y) - \delta \xi)  \right].
	\end{split}
\end{equation}

Since Hamiltonian is a linear function in $\xi$, the existence of optimal control is guarenteed by \textit{Filippov existence theorem} \cite{cesari2012optimization}.  Since we are minimizing the Hamiltonian, the optimal strategy will be a combination of bang-bang and singular controls which is given by 

\begin{equation}
	\xi^*(t) =
	\begin{cases}
		\xi_{\max}, &\text{ if } \frac{\partial \mathbb{H}}{\partial \xi} < 0.\\
		\xi_{\min}, &\text{ if } \frac{\partial \mathbb{H}}{\partial \xi} > 0.
	\end{cases}
\end{equation}
where
\begin{equation}
	\frac{\partial \mathbb{H}}{\partial \xi} = p x \alpha \left(1 - \frac{x}{\gamma}\right) (\omega x^2 + 1) + q \delta y (\omega x^2 + 1) - q \alpha m y (\omega x^2 + 1).
\end{equation}

This problem (\ref{4midxi}) admits a singular solution if there exists an interval $[s_1,s_2]$ on which $\frac{\partial \mathbb{H}}{\partial \xi} = 0$. Therefore, the solution is a combination of bang-bang and singular controls.

\subsection{Applications to Pest Management}

In this subsection, we simulated the time-optimal control problems (\ref{4midalpha}) and (\ref{4midxi}) using CasADi in python \cite{CasADi}. We implemented the direct transcription method with multiple shooting in order to solve the time-optimal control problems. In this method, we discretize the control problem into smaller intervals using finite difference integration, specifically the fourth-order Runge-Kutta (RK4) method. By breaking the trajectory into multiple shooting intervals, the state and control variables at each node are treated as optimization variables. The dynamics of the system are enforced as constraints between nodes, allowing for greater flexibility and improved convergence when solving the nonlinear programming problem with CasADi’s solvers.

\begin{figure}[!ht]
	\centering
	\includegraphics[width=\textwidth]{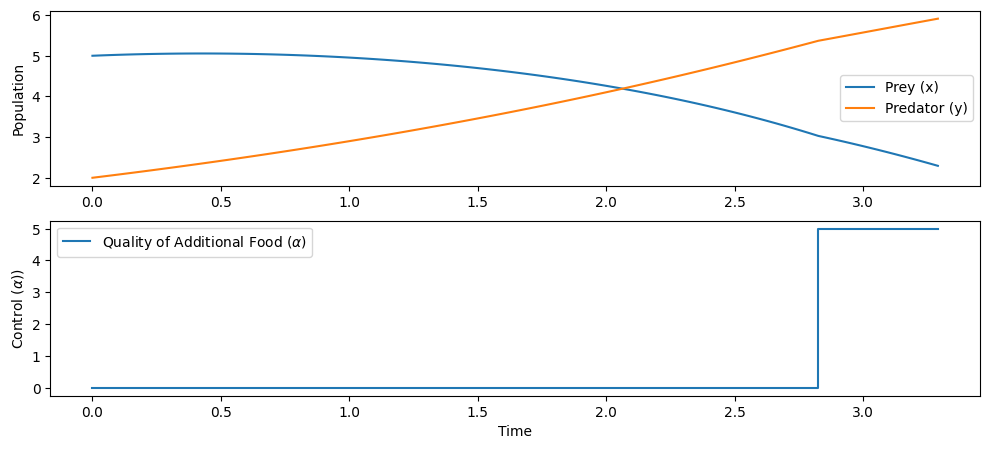}
	\caption{The optimal state trajectories and the optimal control trajectories for the time optimal control problem (\ref{4midalpha}).}
	\label{c4midalpha}
\end{figure}

\begin{figure}[!ht]
	\centering
	\includegraphics[width=\textwidth]{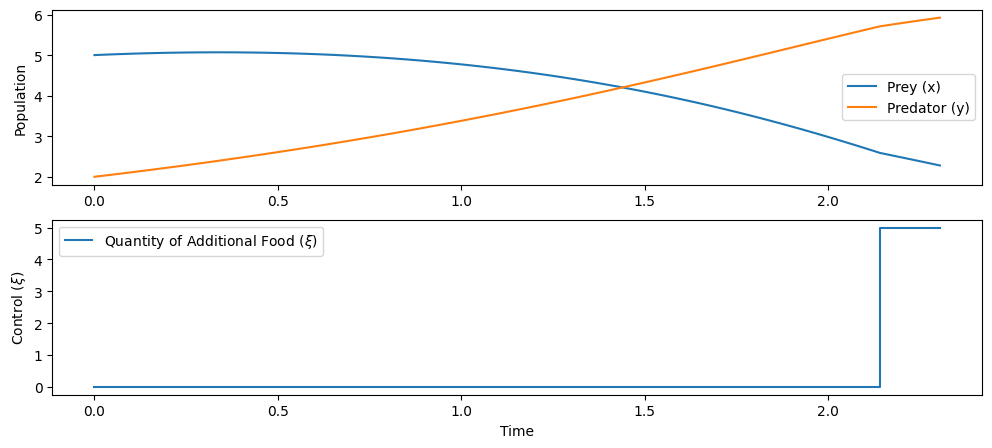}
	\caption{The optimal state trajectories and the optimal control trajectories for the time optimal control problem (\ref{4midxi}).}
	\label{c4midxi}
\end{figure}

\autoref{c4midalpha} illustrates the optimal state and control trajectories for the time-optimal control problem (\ref{4midalpha}). The simulation uses the parameter values $\gamma = 8.0,\  \xi = 0.1,\ \delta = 0.96, \ m=0.3,\ \epsilon = 0.1,\ \omega = 0.1$, starting from the initial point $(5,2)$ and reaching the final point $(1,4)$ in an optimal time of $3.29$ units. 

\autoref{c4midxi} illustrates the optimal state and control trajectories for the time-optimal control problem (\ref{4midxi}). The simulation uses the parameter values $\gamma = 10.0,\  \alpha = 0.2,\ \delta = 0.96, \ m=0.3,\ \epsilon = 0.1,\ \omega = 0.01$, starting from the initial point $(5,2)$ and reaching the final point $(1,4)$ in an optimal time of $2.30$ units. 

\section{Stochastic Model Formulation} \label{sec:4mismodel}

We derived the following prey-predator system involving Holling type-IV functional response and mutual interference among predators in (\ref{4midtemp}).
\begin{equation*}
	\begin{split}
		\frac{\mathrm{d}N}{\mathrm{d}T} & = r N \left(1-\frac{N}{K}\right) -  \frac{cNP}{(bN^2+1)(a+\alpha \eta A+ \epsilon_1 a P)+N}, \\
		\frac{\mathrm{d}P}{\mathrm{d}T} & = \frac{\delta_1 \left(N+\eta A(bN^2+1)\right) P}{(bN^2+1)(a+\alpha \eta A+ \epsilon_1 a P)+N} - m_1 P. \\
	\end{split}
\end{equation*}

In order to make the prey-predator system more realistic, we incorporated the environmental fluctuations by perturbing the prey growth rate and predator death rate by noises. 
\begin{equation} \label{noise4mis}
	r \rightarrow r + \sigma_1 dW_1(t),\  m \rightarrow m +\sigma_2 dW_2(t).
\end{equation}

where $W_i(t) \ (i = 1, 2)$ are the mutually independent standard Brownian motions with $ W_i(0) = 0$ and  $\sigma_1$ and $\sigma_2$ are positive constants and they represent the intensities of the white noise. Then the deterministic system (\ref{4mid}) is modified to the following stochastic system.

Then the deterministic system (\ref{4mid}) is modified and non-dimensionalized to the following stochastic system.
\begin{equation}\label{4mis0}
	\begin{split}
		\mathrm{d} x(t) & = x(t) \left[ 1-\frac{x}{\gamma} - \frac{y}{(1+\alpha \xi+\epsilon y)(\omega x^2 + 1) + x}\right] \mathrm{d}t + \sigma_1 x(t) \mathrm{d}W_1(t), \\
		\mathrm{d} y(t) & = y(t) \left[\frac{\delta(x + \xi (\omega x^2 + 1))}{(1 + \alpha \xi + \epsilon y)(\omega x^2 + 1) + x} - m \right] \mathrm{d}t + \sigma_2 y(t) \mathrm{d}W_2(t). \\
	\end{split}
\end{equation}

Also, the system can go through huge, occasionally catastrophic disturbances. Since white noise is a continuous noise, it cannot capture sudden environmental changes. To cater to these, we also apply a discontinuous stochastic process as compensated Poisson process to model these abrupt natural phenomenon as in \cite{la2010dynamics, Levyjumps}. We now perturb $r$ and $m_1$ with discontinuous noise in addition to the continuous white noise. So, we have
\begin{equation} \label{4misnoise1}
	r \rightarrow r + \sigma_1 \mathrm{d}W_1(t) + \int_{\mathbb{Y}}^{} \gamma_1 (v) \, \widetilde{N} (\mathrm{d}t,\mathrm{d}v) ,\  -m_1 \rightarrow - m_1 + \sigma_2 \mathrm{d}W_2(t) + \int_{\mathbb{Y}}^{} \gamma_2 (v) \, \widetilde{N} (\mathrm{d}t,\mathrm{d}v).
\end{equation}

According to the L\'evy decomposition theorem \cite{oksendal2005stochastic}, we have
$\widetilde{N} (t,dv) = N(t,dv) - \lambda (dv)t$, where $\widetilde{N}(t,dv)$ is a compensated Poisson process and N is a Poisson counting measure with characteristic measure $\lambda$ on a measurable subset $\mathbb{Y}$ of $(0,+\infty)$ with $\lambda(\mathbb{Y}) < \infty$. The distribution of L\'evy jumps L$_i$(t) can be completely parameterized by $(a_i,\sigma_i,\lambda)$ and satisfies the property of infinite divisibility.

Then the deterministic system is modified to the following system of stochastic differential equations: 

\begin{equation}\label{4mis}
	\begin{split}
		\mathrm{d} x(t) =& x(t) \left[ 1-\frac{x}{\gamma} - \frac{y}{(1+\alpha \xi+\epsilon y)(\omega x^2 + 1) + x}\right] \mathrm{d}t \\
		& + \sigma_1 x(t) \mathrm{d}W_1(t) + x(t) \int_{\mathbb{Y}}^{} \gamma_1 (v) \, \widetilde{N} (\mathrm{d}t,\mathrm{d}v), \\
		\mathrm{d} y(t) =& y(t) \left[\frac{\delta(x + \xi (\omega x^2 + 1))}{(1 + \alpha \xi + \epsilon y)(\omega x^2 + 1) + x} - m \right] \mathrm{d}t \\
		&+ \sigma_2 y(t) \mathrm{d}W_2(t) + y(t) \int_{\mathbb{Y}}^{} \gamma_2 (v) \, \widetilde{N} (\mathrm{d}t,\mathrm{d}v).  \\
	\end{split}
\end{equation}

\begin{thm} \label{4misth1}
	For any given initial value $X(\theta) = (x(\theta),y(\theta)) \in C([-\tau_0,0],\mathbb{R}^{+^2})$, there exists a unique positive global solution $(x(t),y(t))$ of system (\ref{4mis}) on $t \geq 0$.
\end{thm}

\textbf{Note: }The above theorem for existence of solutions of (\ref{4mis}) can be proved in similar lines to the proof in \cite{Levyjumps} using the Lyapunov method.

\section{Stochastic Time-optimal Control Studies}\label{sec:4miscontrol}

A stochastic time-optimal control problem involves determining a control strategy that drives a dynamical system to a desired target state in the least expected time while accounting for randomness in the system. Unlike deterministic models, these systems are influenced by probabilistic elements such as random environmental disturbances or sudden unpredictable events, which are typically modeled using stochastic differential equations (SDEs). 

Our goal is to minimize the expected hitting time to a target state while keeping the system’s trajectory stable and within feasible bounds under uncertainty. However, in this work, we minimize not only the expected time it takes for the prey population to fall below a target threshold but also to penalize the overuse of control resources. Incorporating extra terms beyond just the final hitting time in the cost function is essential for maintaining practical viability and system stability. If the cost depended solely on minimizing the time to reach a target, the optimal control might become unrealistically aggressive or discontinuous. By adding penalty terms for the magnitude of control inputs, the optimization naturally discourages extreme or biologically infeasible control strategies. These additional terms promote smoother and more balanced controls, making the resulting policy both efficient and implementable in real-world scenarios, where excessive intervention may not be desirable or possible.

In this section, we solve the stochastic time-optimal control problem with the state space (\ref{4mis}) given by

\begin{align*}
	\begin{split}
		\mathrm{d} x(t) =& x(t) \left[ 1-\frac{x}{\gamma} - \frac{y}{(1+\alpha \xi+\epsilon y)(\omega x^2 + 1) + x}\right] \mathrm{d}t \\
		& + \sigma_1 x(t) \mathrm{d}W_1(t) + x(t) \int_{\mathbb{Y}}^{} \gamma_1 (v) \, \widetilde{N} (\mathrm{d}t,\mathrm{d}v), \\
		\mathrm{d} y(t) =& y(t) \left[\frac{\delta(x + \xi (\omega x^2 + 1))}{(1 + \alpha \xi + \epsilon y)(\omega x^2 + 1) + x} - m \right] \mathrm{d}t \\
		&+ \sigma_2 y(t) \mathrm{d}W_2(t) + y(t) \int_{\mathbb{Y}}^{} \gamma_2 (v) \, \widetilde{N} (\mathrm{d}t,\mathrm{d}v).  \\
	\end{split}
\end{align*}

Here, we consider both quality of additional food ($\alpha$) and quantity of additional food ($\xi$) as control parameters. We seek controls $(\alpha(t), \xi(t))$ minimizing the expected time \( T \) until the prey population falls below a target threshold $x_{\text{target}}$, subject to constraints:
$\alpha_{\text{min}} \leq \alpha(t) \leq \alpha_{\text{max}}, \quad \xi_{\text{min}} \leq \xi(t) \leq \xi_{\text{max}}.$

The cost functional is:
$$ J(\alpha(\cdot), \xi(\cdot)) = \mathbb{E}\left[ T + \int_0^T \left( \rho_{\alpha} \alpha^2(t) + \rho_{\xi} \xi^2(t) \right) dt \right],$$

where \( \rho_{\alpha}, \rho_{\xi} \) are small positive constants penalizing excessive control efforts.

The Hamiltonian \( H \) is:

\begin{align*}
	H(x, y, \alpha, \xi, p_1, p_2) &= \rho_{\alpha} \alpha^2 + \rho_{\xi} \xi^2 + p_1 x \left( 1-\frac{x}{\gamma} - \frac{y}{(1+\alpha \xi+\epsilon y)(\omega x^2 + 1) + x} \right) \\
	&\quad + p_2 y \left(\frac{\delta(x + \xi (\omega x^2 + 1))}{(1 + \alpha \xi + \epsilon y)(\omega x^2 + 1) + x} - m \right).
\end{align*}

where $p_1(t), p_2(t)$ are the adjoint (costate) processes.

The adjoint equations are:

\begin{align*}
	dp_1(t) &= -\left( \frac{\partial H}{\partial x} \right) dt, \\
	dp_2(t) &= -\left( \frac{\partial H}{\partial y} \right) dt,
\end{align*}
with terminal conditions:
$ p_1(T) = 0, \quad p_2(T) = 0.$

Explicitly, the partial derivatives are:

\begin{align*}
	\frac{\partial H}{\partial x} =& p_1 \left[ 1 - \frac{ 2 x}{\gamma} - \frac{(1 - \omega x^2) (1+\alpha \xi + \epsilon y) y }{\left( (1 + \alpha \xi + \epsilon y) (\omega x^2 + 1) + x \right)^2}\right]  + \delta p_2 y \frac{(1 - \omega x^2) (1+\alpha \xi + \epsilon y - \xi)}{\left( (1 + \alpha \xi + \epsilon y) (\omega x^2 + 1) + x \right)^2}, \\
	\frac{\partial H}{\partial y} =& \frac{-p_1 x \left((1+\alpha \xi) (\omega x^2 + 1) + x\right)}{\left( (1 + \alpha \xi + \epsilon y) (\omega x^2 + 1) + x \right)^2} + p_2 \left[ \frac{\delta (x+\xi (\omega x^2 + 1)) ((1 + \alpha \xi) (\omega x^2 + 1)+x)}{\left( (1 + \alpha \xi + \epsilon y) (\omega x^2 + 1) + x \right)^2} - m\right]
\end{align*}

The controls $( \alpha(t), \xi(t) )$ are updated by minimizing the Hamiltonian pointwise:

\begin{align*}
	\frac{\partial H}{\partial \alpha} =& 2 \rho_{\alpha} \alpha - \frac{\xi (\omega x^2 + 1)}{((1+\alpha \xi + \epsilon y)(\omega x^2 + 1)+x)^2} \\
	& \left[\delta p_2 y (x + \xi (\omega x^ 2 + 1)) - p_1 x y \right], \\
	\frac{\partial H}{\partial \xi} =& 2 \rho_{\xi} \xi + \frac{y (\omega x^2 + 1)}{((1+\alpha \xi + \epsilon y)(\omega x^2 + 1)+x)^2} \\
	& \left[\alpha p_1 x + \delta p_2 \left((1-\alpha) x + (1 + \epsilon y) (\omega x^2 + 1)\right)\right].
\end{align*}

The optimal control values are the solutions of the above equations equating to zero. We now numerically simulate this stochastic time optimal control in Python using Monte Carlo estimation. This model includes two types of stochastic disturbances: continuous and jump-type. The continuous noise is modeled by Brownian motion, capturing small, persistent environmental fluctuations such as temperature or resource variability. The jump-type noise is modeled using compensated Poisson processes, representing sudden events like natural disasters or abrupt food supply changes. Theoretically, these are incorporated via diffusion and jump terms in the (\ref{4mis}). In numerical simulations in Python, Brownian noise is implemented using scaled normal distributions (np.random.randn()), while Poisson jumps are simulated using random samples from Poisson distributions (np.random.poisson()), with compensation to maintain zero mean. Both types of noise enrich the realism and robustness of the model.

We implemented the Monte Carlo average to the solutions of this stochastic time optimal control problem. A Monte Carlo average is a statistical method that estimates expected outcomes by running a model repeatedly with random inputs and averaging the results. In the context of stochastic systems, it allows for approximation of probabilistic quantities, such as expected time to reach a target or mean trajectory behavior. Each simulation represents a possible realization of the system under uncertainty, and the average over many such paths provides a reliable estimate of the true expected outcome. This method is particularly valuable when analytical solutions are intractable due to the system’s complexity or randomness.

Using a Monte Carlo average over $1000$ simulations in this model ensures statistical reliability in estimating the expected time to reach the target predator level. Given the stochastic nature of the prey-predator dynamics, each run yields a different trajectory due to random fluctuations and jumps. A single run is not representative; averaging over a large number of simulations smooths out random variation and yields a robust estimate of performance under uncertainty. The choice of $1000$ simulations strikes a balance between computational efficiency and statistical accuracy, offering a high-confidence approximation of the system’s expected behavior under optimal control.

At each simulation, the Gradient descent updates in discretized time are given by

$$\alpha^{\text{new}} = \text{Proj}_{[\alpha_{\min}, \alpha_{\max}]}\left( \alpha - \eta_1 \left( \frac{\partial H}{\partial \alpha} \right) \right),$$

$$\xi_{\text{new}} = \text{Proj}_{[\xi_{\min}, \xi_{\max}]}\left( \xi - \eta_2 \left( \frac{\partial H}{\partial \xi} \right) \right),$$

where \( \eta_1, \eta_2 \) are learning rates, and \( \text{Proj} \) denotes projection onto the admissible control bounds. The simulation stops at the first hitting time \( T \) when: $x(T) \leq x_{\text{target}}.$ 

\autoref{Control4mis} illustrates the optimal state and control trajectories for the stochastic time-optimal control problem. Using parameter values $\gamma = 10.0,\ \epsilon = 0.1,\ \delta = 11.0,\ m = 5.05,\ \omega = 0.1$, the system evolves from an initial state of $(50, 10)$ to a final state of approximately $(1, 70)$. The figure displays the first 10 trajectories from a total of 1,000 Monte Carlo simulations, with the solid curves representing the mean behavior of the populations and control variables over all simulations. The only terminal constraint enforced in this setup is that the prey population must reach the target value of 1. This simulation demonstrates a prey eradication strategy, which is particularly relevant in ecological control and pest management applications, where minimizing prey presence in minimal time is a critical objective.

\begin{figure}[!ht]
	\centering
	\includegraphics[width=\textwidth]{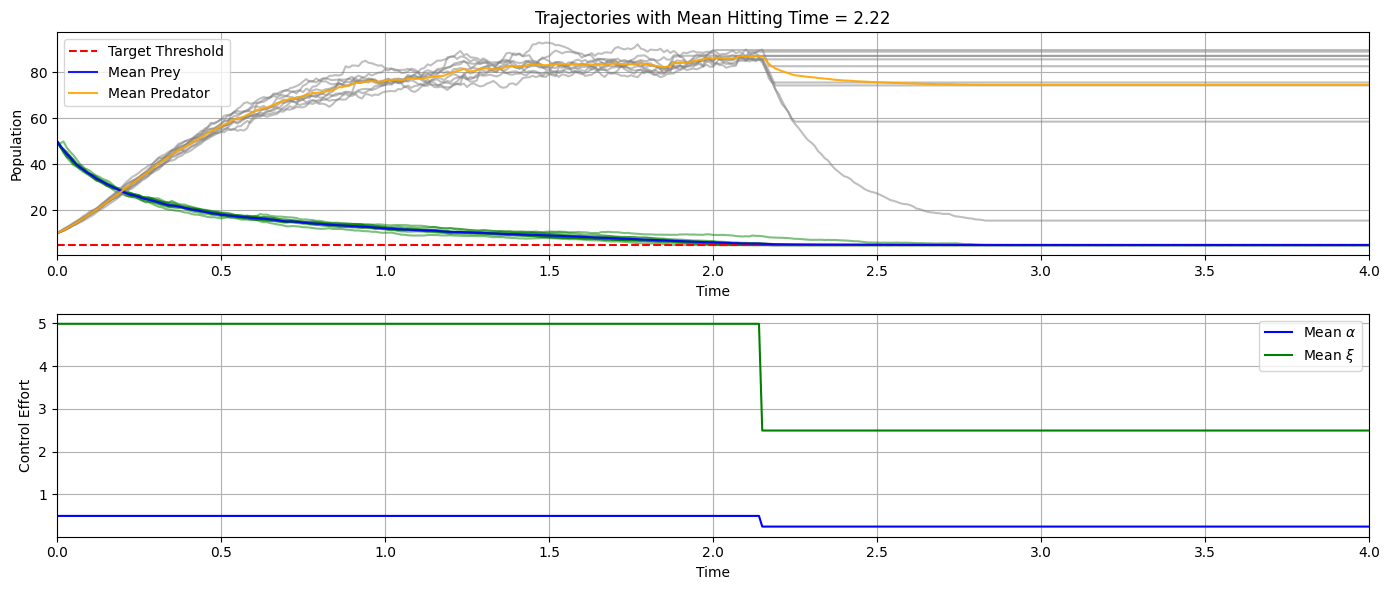}
	\caption{The optimal state trajectories and the optimal control trajectories for the stochastic time optimal control problem with (\ref{4mis}).}
	\label{Control4mis}
\end{figure}

\begin{figure}[!ht]
	\centering
	\includegraphics[width=\textwidth]{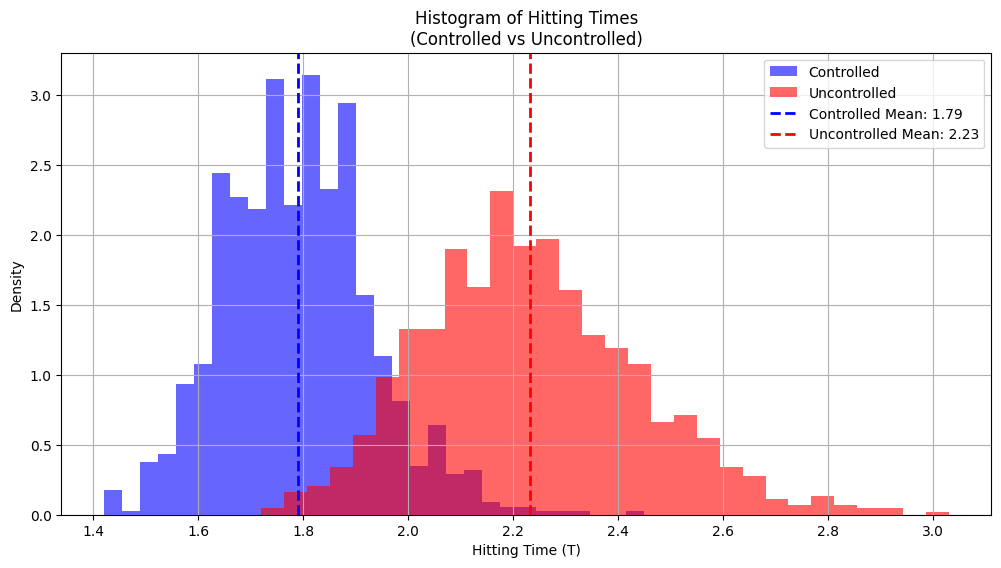}
	\caption{The density of hitting times of the controlled system in comparison with the uncontrolled system.}
	\label{Histogram4mis}
\end{figure}

\begin{figure}[!ht]
	\centering
	\includegraphics[width=\textwidth]{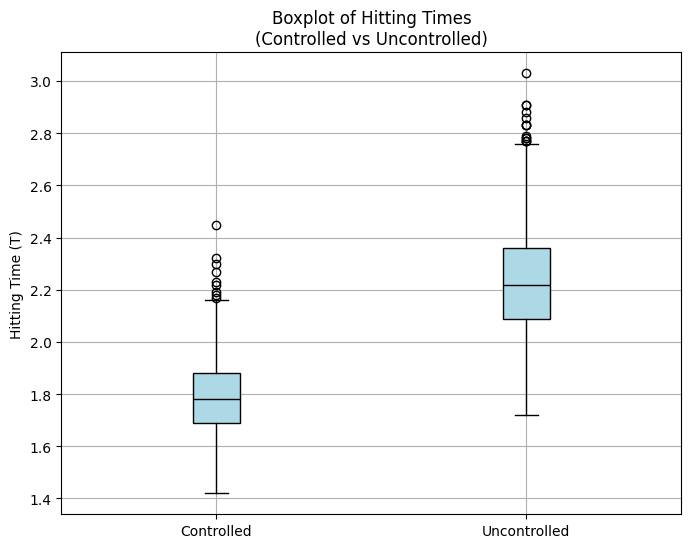}
	\caption{The box plots representing the distribution of hitting times of the controlled system in comparison with the uncontrolled system.}
	\label{Box4mis}
\end{figure}

\begin{table}[h!]
	\centering
	\begin{tabular}{|c|c|c|}
		\hline
		\textbf{Case} & \textbf{Mean Hitting Time} & \textbf{95\% Confidence Interval} \\
		\hline
		Controlled & 1.79 & (1.78, 1.80) \\
		Uncontrolled & 2.23 & (2.22, 2.24) \\
		\hline
	\end{tabular}
	\caption{Comparison of mean hitting times and 95\% confidence intervals for controlled and uncontrolled scenarios.}
	\label{4mid_hitting_times}
\end{table}

The histogram in \autoref{Histogram4mis} and the box plot in \autoref{Box4mis} reveals a clear distinction between the controlled and uncontrolled scenarios in terms of the time it takes for the predator population to reach the target threshold. The distribution of hitting times under control is skewed toward lower values, with a significantly lower mean compared to the uncontrolled case. This indicates that applying optimal control—through the quality and quantity of additional food—effectively accelerates the system toward the desired state. In contrast, without intervention, the system requires substantially more time on average to reach the same target, with wider variability and longer tails in the distribution, as evidenced by the broader histogram and boxplot spread. The narrower interquartile range in the boxplot for the controlled case further indicates that the control not only shortens the expected time but also stabilizes the process, leading to more predictable outcomes. This supports the conclusion that well-designed controls significantly improve both performance and robustness in stochastic ecological systems.

In order to prove that the controlled and uncontrolled samples are not from the same distribution, we calculate the p-vlaue of these two distributions using the Mann–Whitney U test. By comparing the hitting times of the controlled and uncontrolled stochastic systems, we got a p-value of approximately $3.27 \times 10^{-8}$. This value is vanishingly small, far below any conventional significance threshold, providing overwhelming evidence against the null hypothesis. The means of controlled and uncontrolled systems are detailed with the confidence intervals in \autoref{4mid_hitting_times}. The difference in means between the two cases highlights the efficiency and impact of the control strategy. In practical terms, such a result strongly supports the conclusion that the control strategy—based on additional food quantity and quality—significantly reduces the time required for the system to reach the target state. The result not only confirms the statistical significance of the observed improvement but also highlights the robustness and consistency of the control mechanism across stochastic scenarios.

\section{Discussions and Conclusions} \label{sec:disc}

\indent In this paper, we derived the functional response incorporating additional food, prey defence, and mutual interference, and formulated the corresponding prey-predator model. After proving the positivity and boundedness of the solutions of the proposed system, we investigated the conditions for the existence of various equilibria. Later, We presented the local stability analysis of these equilibria and explored the possible local bifurcations exhibited by the model, both analytically and numerically. We further studied the global dynamics of the system in the parameter space defined by the quality and quantity of additional food and provided a detailed analysis of the consequences of providing additional food. In the next part, We presented theoretical and numerical results on time-optimal control problems, using the quality or quantity of additional food as control parameters. These works are further extended to the stochastic system with both continuous and discrete noises. We then formulated the time-optimal control problem and provided analytical and numerical solutions under the influence of two controls and two types of noise. 

These results not only provide a rigorous theoretical understanding of the proposed prey-predator system but also have significant implications for real-world pest management strategies. Biological control (or biocontrol) of pests, which involves the use of natural enemies such as predators, parasitoids, or pathogens to suppress pest populations, offers a sustainable and environmentally friendly alternative to chemical pesticides. Our findings reinforce that the provision of additional food to natural enemies plays a crucial role in enhancing the effectiveness of biocontrol. In particular, we analyzed the role of additional food across four distinct scenarios, ranging from the most effective to the least favorable strategies, to evaluate their potential in achieving long-term pest suppression. 

From section \ref{sec:4midconseq}, we observed that the regions $A_{11},\ A_{21},\ A_{31}$ exhibit bistability, where both pest-free and interior equilibria are stable. In these regions, carefully selecting the initial conditions and appropriately provisioning additional food can steer the system toward the pest-free equilibrium, which is the most desirable outcome for pest management. The regions $A_{12},\ A_{13},\ A_{22},\ A_{23},\ A_{32},\ A_{33}$ are next in desirability, as they support only interior equilibria as stable outcomes. Although pest eradication is not achieved here, the pest population can be maintained at sufficiently low levels that prevent significant crop damage, making this an effective strategy in practice.

The regions $A_{14},\ A_{24},\ A_{34}$ are less favorable, as they permit both predator-free and interior equilibria to be stable. In such cases, pest suppression depends heavily on initial conditions. With careful selection of the starting point and proper food provisioning, it may still be possible to limit crop damage, though the risk of pest outbreaks is higher. The least desirable regions are $A_{15},\ A_{25},\ A_{35}$, where the only stable equilibrium is pest-dominant. Here, the provision of additional food inadvertently promotes pest persistence and even natural enemy extinction, highlighting the dangers of indiscriminate food supplementation.

To conclude, although additional food is non-reproductive, its arbitrary or excessive provision can lead to unintended and adverse ecological consequences, including pest dominance and collapse of the predator population. Importantly, the deterministic and stochastic optimal control studies presented in this work offer valuable guidance for steering the system toward desired equilibrium states by optimizing food quality, quantity, and timing. These results emphasize that successful pest management must be rooted in a careful balance of ecological insight and control strategy design, rather than simplistic or uniform interventions. These insights contribute meaningfully to the design of sustainable and optimized pest management practices.

\subsection*{Financial Support: }
This research was supported by National Board of Higher Mathematics (NBHM), Government of India (GoI) under project grant - {\bf{Time Optimal Control and Bifurcation Analysis of Coupled Nonlinear Dynamical Systems with Applications to Pest Management, \\ Sanction number: (02011/11/2021NBHM(R.P)/R$\&$D II/10074).}}

\subsection*{Conflict of Interests Statement: }
The authors have no conflicts of interest to disclose.

\subsection*{Ethics Statement:} 
This research did not required ethical approval.

\subsection*{Acknowledgments}
The authors dedicate this paper to the founder chancellor of SSSIHL, Bhagawan Sri Sathya Sai Baba. The contributing author also dedicates this paper to his loving elder brother D. A. C. Prakash who still lives in his heart.


\bibliographystyle{plain} 
\bibliography{reference}

\end{document}